\documentclass{m2an}

\newcommand{\C}{\mathbb C}
\newcommand{\Z}{\mathbb Z}

\newcommand{\A}{\mathbb{A}}

\newcommand{\Hom}{\operatorname{Hom}}

\newcommand{\tr}{\operatorname{tr}}

\newcommand{\Map}{\operatorname{Map}}
\newcommand{\Ext}{\operatorname{Ext}}

\newcommand{\trdeg}{\operatorname{trdeg}}
\newcommand{\Sing}{\operatorname{Sing}}

\usepackage[colorlinks,linkcolor=blue,anchorcolor=blue,citecolor=red]{hyperref}
\newtheorem{theorem}{Theorem}[section]
\newtheorem{example}[theorem]{Example}
\newtheorem{corollary}[theorem]{Corollary}
\newtheorem{lemma}[theorem]{Lemma}
\newtheorem{remark}[theorem]{Remark}
\newtheorem{definition}[theorem]{Definition}
\newtheorem{proposition}[theorem]{Proposition}

\usepackage{bm}
\usepackage{float}
\usepackage{multirow}
\usepackage{upgreek}
\usepackage{amsfonts}
\usepackage{amssymb}
\usepackage{graphicx}
\usepackage{mathrsfs}
\usepackage{tabularx}
\usepackage{array}
\usepackage{cite}
\usepackage{amsthm}
\usepackage{amsmath}
\usepackage{algorithm}
\usepackage{algpseudocode}
\usepackage{float}

\begin{document}
%%-----------------------------
%%      the top matter
%%-----------------------------
\title{On Intersection of Character Varieties Associated to Dehn Fillings of the Whitehead Link and its Generalization}
\author{Jiayu Wan}
\address{Institute of Mathematical Sciences,ShanghaiTech University, Shanghai 201210, China. \\ Email:wanjy@shanghaitech.edu.cn }
\author{Charles Katerba}
\address{Montana State University \\ Email:charles.katerba@gmail.com}

\begin{abstract}
The canonical component of the character variety of a hyperbolic 2-bridge link complement $M$ is an algebraic surface, from the work of Thurston. If we fix a boundary component $T$ of $M$ and perform a Dehn-filling on $T$, we get a 1-cusped manifold whose character variety is commonly a generalized curve that is naturally embedded in the character variety of $M$. If we perform distinct Dehn-fillings on $M$, we get various generalized curves on the canonical surface. In this paper, we analyze how these curves intersect with each other in the surface and study the topological implications for these intersection points. We also classify the intersection points that correspond to reducible characters, using the twisted Alexander polynomials associated to the Dehn-fillings. Throughout the paper, we use the Whitehead link as a motivating example to illustrate our core methodology, while some generalizations to other 2-bridge links will also be discussed.
\end{abstract}


\date{}

\maketitle

%-------------------------------------------------
\section{Introduction}
Let $\Gamma$ be a finitely generated group and let $X(\Gamma)$ be the $\operatorname{SL}_2(\mathbb{C})$ character variety of $\Gamma$ (see \S 2.1). When $\Gamma$ is the fundamental group of a hyperbolic 3-manifold $M$, seminal work of Thurston and Culler-Shalen(\cite{Culler_Shalen}) established $X(\Gamma)$ as a powerful tool in the study of the geometry and topology of $M$. For a hyperbolic 3-manifold $M$, the dimension of a canonical component(cf. \S 2.2) of the character variety is equal to its number of cusps. In particular, if $M$ is a hyperbolic link complement $\mathbb{S}^3$ whose boundary consists of two tori, a canonical component of $X(\Gamma)$ will be an algebraic surface. This surface contains a wealth of information about $M$ and is also valuable as a subject of study in its own right(\cite{MPV2011,Landes,Harada,PT2015}).

Let $M$ be a 2-bridge link complement in $\mathbb{S}^3$ and fix one of the boundary components $T$. We can perform a Dehn filling of slope $\gamma=\frac{p}{q}$ on $T$ to obtain a one-cusped 3-manifold $M_{\gamma}$. This Dehn filling will induce an algebraic subvariety $C_{\gamma}$ of the canonical surface(see \S 2.3). In most situations(e.g. when the Dehn filling is also hyperbolic), $C_\gamma$ will be an algebraic curve(not necessarily irreducible) with possibly isolated 0-dimensional components. Also, since both of the boundary components for a 2-bridge link $M$ are unknots, the isomorphism class of $C_\gamma$ will be independent of the boundary torus $T$ we pick. The main focus of the paper is to study how $C_\gamma$ intersects with $C_{\delta}$, when two distinct Dehn filling slopes $\gamma$ and $\delta$ are considered. The significance of studying such intersection pattern is at least twofold. First, the intersection points may carry important topological information of the corresponding 3-manifolds(cf. \S 4).  Second, it is a standard paradigm in algebraic geometry to study the intersection of two curves on an algebraic surface and read off structural information about the surface itself from that intersection. Therefore, it is reasonable to expect that we can extract useful information about the canonical surface from the intersection $C_\gamma \cap C_\delta$. This paper aims to lay the groundwork for future research in this direction and to clarify some potentially useful concepts and theories.

We begin with a systematic analysis of the intersection pattern $C_\gamma \cap C_\delta$ for two distinct slopes $\gamma$ and $\delta$. This includes finding the common 1-dimensional components, locating the common isolated points, and computing the intersection number and multiplicity for the pure 1-dimensional components of $C_\gamma$ and $C_\delta$. In this paper, we use the Whitehead link complement in $\mathbb{S}^3$ as a motivating example, and the generalization to other 2-bridge links will also be discussed. The main contribution of the paper is that we come up with an algorithm in the case of the Whitehead link(Theorem \ref{Algorithm for computing intersection pattern}) that completely classifies the intersection pattern, once two slopes $\gamma$ and $\delta$ are given as inputs. The algorithm is easy to implement, and provides a way to decouple the complicated problem into simpler and independent pieces. The algorithm can also be generalized and applied to other 2-bridge links(Theorem \ref{Algorithm for computing intersection pattern in general}), provided that some preliminary information of the associated character variety has been computed \textit{a priori}. The key idea underlying our approach is that we consider the boundary restriction map $\phi$ from the canonical component of the given character variety $X_0(M)$ onto the character variety of the corresponding boundary torus $X(T)$. We push forward $C_\gamma$ and $C_\delta$ downstairs to $X(T)$, find the intersection pattern there(which is much easier to compute), and then pull back the intersection pattern upstairs using the properties of $\phi$. We also provide a wide range of examples for different choices of slopes that exhibit completely different intersecting behavior. 

Next, analyze the topological implications for an intersection point in $p \in C_\gamma \cap C_\delta$. In fact, when one of the coordinates of $p$ lies in some number field $K$ but is not an algebraic integer(such a character $p$ is called the character of an ANI representation), the point detects a common boundary slope for both the Dehn filled manifold $M_\gamma$ and $M_\delta$. We explain how this works in terms of Culler-Shalen theory. In the Whitehead link case, we provide a simple criterion to decide whether an intersection point has an ANI coordinate(Proposition \ref{criterion for ANI}), and we also compute the common boundary slopes that could possibly arise in this way using purely arithmetic arguments(Theorem \ref{boundary slopes are 0 or 4}). This last observation is in accordance with the common boundary slopes computed from the topological method of Hatcher and Thurston(\cite{HT1985}). 

Finally, since reducible and irreducible characters typically exhibit very different behaviors, it is natural to ask when an intersection point corresponds to a reducible character. We give a characterization of a reducible intersection point in terms of common roots of twisted Alexander polynomials. As a corollary of the characterization, we show that the canonical curves associated with different twist knots arising from Dehn surgeries of the Whitehead link can never intersect at a reducible character. 

The rest of the paper is organized as follows. In section 2, we begin with a short introduction to character varieties of cusped hyperbolic 3-manifolds and the associated Dehn fillings. In section 3, we study the intersection pattern associated to Dehn filled manifolds in the Whitehead link complement and discuss its generalization to other 2-bridge links. In section 4, we provide a topological implication of the intersection points and analyze how they can be used to detect common boundary slopes. Finally in section 5, we give a characterization of reducible intersection point in terms of twisted Alexander polynomials. 

\bigskip

\noindent \textbf{Role of AI.} During the preparation of this manuscript, the author used OpenAI's ChatGPT (GPT-5.6) as an assistive tool for language editing, organization and presentation, literature searches, translation, and related aspects of manuscript preparation. In particular, Lemmas \ref{trace functions of n and l} and \ref{finite locus} were computed with the assistance of AI. The author reviewed all resulting material and takes full responsibility for the mathematical content and the final manuscript. 
\section{Character varieties and cusped hyperbolic 3-manifolds}

\subsection{Character Varieties}
Character varieties have proven to be a powerful tool in the study of hyperbolic 3-manifolds, bridging the algebraic properties of the fundamental groups and the topology of the 3-manifolds. For important works in this direction, we refer the reader to \cite{Culler_Shalen, Shalen, CGLS1987, CCGLS1994}. In this subsection, we present a short introduction to character varieties. Let $\Gamma$ be a finitely generated group, and let $R(\Gamma)$ be the set of all representations of $\Gamma$ in $SL_2(\mathbb{C})$, i.e. $R(\Gamma)=\Hom(\Gamma, \operatorname{SL}_2(\mathbb{C}))$. As discussed in \cite{Culler_Shalen}, $R(\Gamma)$ is naturally endowed with the structure of an affine algebraic closed subset of $\mathbb{C}^{4n}$, and we call $R(\Gamma)$ (with this algebraic structure) the representation variety of $\Gamma$. Two elements $\rho_1, \rho_2 \in R(\Gamma)$ are called equivalent if there exists $g \in GL_2(\mathbb{C})$ such that $\rho_2 = g\rho_{1}g^{-1}$. A representation $\rho \in R(\Gamma)$ is called irreducible if there is no nontrivial subspace of $\mathbb{C}^2$ invariant under the action of $\rho(\Gamma)$, and is called reducible if it is not irreducible.

For any $\rho \in R(\Gamma)$, the character of $\rho$ is the function $\chi_{\rho}: \Gamma \longrightarrow \mathbb{C}$ defined by $\chi_{\rho}(g) = tr(\rho(g))$. It is well-known that if $\rho_1,\rho_2 \in R(\Gamma)$ such that $\rho_1$ is irreducible and $\chi_{\rho_1} = \chi_{\rho_2}$, then $\rho_1,\rho_2$ are equivalent. Conversely, traces are always invariant under conjugacy (hence equivalence). So characters 'almost' classify representations up to equivalence. Let $X(\Gamma)$ be the set of all characters of $\Gamma$ in $\operatorname{SL}_2(\mathbb{C})$, then $X(\Gamma)$ is also endowed with an algebraic structure using the arguments in \cite{Culler_Shalen}. We therefore have the following definition. 

\begin{definition}
    Let $\Gamma$ be a finitely generated group and $X(\Gamma)$ the set of characters of $\Gamma$ in $\operatorname{SL}_2(\mathbb{C})$. Then $X(\Gamma)$, together with the endowed algebraic structure, is called the character variety of $\Gamma$ in $\operatorname{SL}_2(\mathbb{C})$.
\end{definition}

There is a natural surjective map $t: R(\Gamma) \longrightarrow X(\Gamma)$ that sends each $\rho \in R(\Gamma)$ to its character $\chi_{\rho}$. It can be shown that $t$ is a regular map, if $R(\Gamma)$ and $X(\Gamma)$ are equipped with the corresponding algebraic structure.

\subsection{Cusped Hyperbolic 3-Manifolds}
We will be particularly interested in the case where $\Gamma=\pi_{1}(M)$ for a $M$ a hyperbolic 3-manifold. A hyperbolic 3-manifold will is a connected, oriented and complete manifold $M$ of the form $\mathbb{H}^3/\Gamma$, where $\mathbb{H}^3$ is the hyperbolic 3-space and $\Gamma \cong \pi_1(M)$ is a torsion free discrete subgroup of $\operatorname{Isom}^{+}(\mathbb{H}^3) \cong \operatorname{PSL}_2(\mathbb{C})$. We refer the reader to \cite{Ratcliffe} and \cite{Thurston_2} for more detail on hyperbolic 3-manifolds. 

We will focus on the situation where $M$ is a knot or link exterior in $\mathbb{S}^3$. In this case, $M$ is of finite volume, and has a finite union of incompressible tori as its boundary. By Mostow Rigidity Theorem(see \cite{Ratcliffe}, \S 11.8), there is a unique (up to equivalence) discrete and faithful representation $\rho_0: \Gamma \longrightarrow \operatorname{PSL}_2(\mathbb{C})$, coming from the holonomy of the complete structure on $M$. This $\rho_0$ lifts to a representation $\hat{\rho_0}: \Gamma \longrightarrow \operatorname{SL}_2(\mathbb{C})$ according to the study of Thurston(\cite{Culler_Shalen}, 3.1.1.). In general, this lift won't be unique even up to equivalence. 

We define a canonical component $X_0(\Gamma) \subset X(\Gamma)$ to be an irreducible component of $X(\Gamma)$ containing the character of some $\hat{\rho_0}$ describe above, i.e. the character of a lift of the discrete faithful representation $\rho_0: \Gamma \longrightarrow \operatorname{PSL}_2(\mathbb{C})$. Notice that there may be multiple canonical components of $X(\Gamma)$ as the lift won't be unique. Any such canonical component is clearly a complex affine variety. Its dimension is given by the number of cusps of the manifold (in our case, it's the number of components of the link) as proved by Thurston in the following important theorem. 

\begin{theorem}\label{Thurston's theorem}
    (\cite{Culler_Shalen}, 3.2.1.) Let $M$ be a noncompact hyperbolic 3-manifold of finite volume with $d$ cusps and set $\Gamma=\pi_1(M)$. Then any canonical component $X_0(\Gamma)$ is a $d$-dimensional complex affine algebraic variety. In particular, when $M$ is a one (resp. two)-cusped hyperbolic 3-manifold, the canonical component $X_0(\Gamma)$ is a complex affine curve (resp. surface). 
\end{theorem}

\subsection{Character Varieties Associated to Dehn Fillings}
Let $M$ be a hyperbolic 3-manifold in $\mathbb{S}^3$ with two cusps, and let $G=\pi_1(M)$. Let $X_0(M)$ be a fixed canonical componnent of $X(M)$, which is an algebraic surface by Theorem \ref{Thurston's theorem}. Denote the two cusps by $T_1,T_2$ and we fix a cusp $T=T_1$ with a preferred pair of meridian and longitude $(\mu,\lambda)$. We perform a Dehn filling on $T$ with slope $\gamma=\frac{p}{q}$ to get an one-cusped manifold $M_\gamma$ in a possibly lens space(when $p\neq1)$. Let $G_{\gamma}:=\pi_1(M_{\gamma})=G/\langle\langle\gamma\rangle\rangle=G/\langle\langle \mu^p\lambda^q\rangle\rangle$. The surjective quotient map $G\longrightarrow G_\gamma$ will induce a closed embedding $j_{\gamma}$: $X(M_{\gamma}) \to X(M)$. Define $C_{\gamma}=j_{\gamma}(X(M_{\gamma})) \cap X_0(M)$. 

From the definition, $C_\gamma$ is an algebraic subvariety of $X_0(M)$. Clearly $C_\gamma \neq X_0(M)$ and hence $C_\gamma$ is a subvariety of dimension at most 1. If $\gamma$ is the slope of a hyperbolic Dehn filling, then by Theorem \ref{Thurston's theorem}, $C_\gamma$ will contain at least one irreducible curve with possibly some isolated 0-dimensional components. We call such a variety a generalized curve in $X_0(M)$. Throughout the paper,we assume $C_\gamma$ is a generalized curve. For two distinct slopes $\gamma \neq \delta$, we are interested in studying the intersection pattern $C_\gamma \cap C_\delta$ in the surface $X_0(M)$. 
\section{intersection pattern}
In this section, we study the intersection $C_\gamma \cap C_\delta$ defined in $\S 2.3$. This includes finding the common 1-dimensional components, computing the intersection number and multiplicity, etc. We begin with a short introduction to some fundamental concepts in algebraic geometry that will be used later in the construction of the algorithm. We then prove and compute some of the important properties of the boundary restriction map in the case of the Whitehead link. Based on these properties, we construct the algorithm which completely classifies the intersection patterns for the Whitehead link and we also discuss the generalization of the algorithm to other 2-bridge links. Finally, we provide a wealth of examples to demonstrate a broad spectrum of possibilities for the intersection patterns.

\subsection{preliminaries}\label{prelimaries 3.1}
We assume the readers are familiar with the basic concepts in algebraic geometry, such as varieties, schemes, dimension of schemes, function fields, morphism and rational maps etc. In this paper, we will only consider affine schemes of the form $X=\operatorname{Spec}(A)$ where $A=\C[x_1,x_2,...,x_n]/I$ and $I$ is some prime ideal of $\C[x_1,x_2,...,x_n]$. The closed points of these schemes typically correspond to the points on some particular component of the character varieties under consideration. In this subsection, we will discuss some fundamental concepts and results in algebraic geometry with respect to morphisms between schemes. We will only give the definitions in the setting of affine schemes mentioned above. For a more general treatment, we refer the reader to \cite{Hartshorne}, \cite{Shafarevich}, \cite{Poonen} and \cite{Milne1980} for more details. We fix the following setup: Let $X=\operatorname{Spec}(A),~ Y=\operatorname{Spec}(B),$ where $A$ and $B$ are integral domains that are also finitely generated $\mathbb C$-algebras, and let $\phi:X\longrightarrow Y$ be a morphism induced by a homomorphism $B\to A$. We assume that $\dim X=\dim Y$. We start with the following definitions. 

\begin{definition}\label{various morphisms}
Consider $\phi: X \longrightarrow Y$, then 

\begin{enumerate}
\item \textbf{Generically finite.}
Suppose that $\phi$ is dominant, i.e. $\overline{\phi(X)}=Y$. The morphism $\phi$ is \emph{generically finite} if the induced extension of function fields
$\C(Y)=\operatorname{Frac}(B)\ \longrightarrow \C(X)=\operatorname{Frac}(A)$
is finite. In our case, under the dominance assumption, the condition $\dim X=\dim Y$ already implies that $\phi$ is generically finite. When $\phi$ is generically finite, the degree of $\phi$ is defined to be $[\C(X):\C(Y)]$. 

\item \textbf{Finite.}
The morphism $\phi$ is \emph{finite} if $A$ is a finitely generated
$B$-module.

\item \textbf{Finite flat.}
The morphism $\phi$ is \emph{finite flat} if $A$ is a finite and
flat $B$-module.

\item \textbf{Finite étale.}
The morphism $\phi$ is \emph{finite étale} if it is finite flat and,
for every closed point $y\in Y$, the scheme-theoretic fiber $X_y=\operatorname{Spec}\bigl(A\otimes_B\kappa(y)\bigr)$ is reduced. Since $\kappa(y)=\mathbb C$, this means that $A\otimes_B\kappa(y)\cong\mathbb C^{\,r}$
for some integer $r\geq 0$. 
\end{enumerate}
\end{definition}

It is more convenient to interpret these various notions of morphisms in terms of fibers of $\phi$. To this end, we need the following definitions.

\begin{definition}\label{def of fiber multiplicity and total length}
Let $y\in Y$ be a closed point, corresponding to a
maximal ideal $\mathfrak m_y\subset B$, and suppose that the
set-theoretic fiber $\phi^{-1}(y)$ is finite. Set $R_y:=A/\mathfrak m_yA.$ For $x\in\phi^{-1}(y)$, let $\mathfrak m_x\subset A$ be its
corresponding maximal ideal. The \emph{multiplicity of $x$ in the
fiber over $y$} is
$$
e_x(y):=\dim_{\mathbb C}(R_y)_{\mathfrak m_x/\mathfrak m_yA}.
$$
The \emph{total length of the fiber over $y$} is
$$
L(y):=\sum_{x\in\phi^{-1}(y)}e_x(y)
     =\dim_{\mathbb C}R_y.
$$
If the fiber is empty, we set $L(y)=0$.
\end{definition}

We now have the following proposition, which translates the various notions of morphisms in terms of the fibers of $\phi$(see \cite{Milne1980}, Chapter I for more references). 

\begin{proposition}\label{characterization of morphisms using fiber}
For each closed point $y\in Y$, consider the set-theoretic fiber $\phi^{-1}(y)$. When this fiber is finite, let $e_x(y)$ denote the multiplicity of $x\in\phi^{-1}(y)$, and set $L(y)=\sum_{x\in\phi^{-1}(y)}e_x(y).$ Then:
\begin{enumerate}
\item $\phi$ is finite if and only if it is proper and
      $\phi^{-1}(y)$ is a finite set for every closed point $y\in Y$.

\item $\phi$ is finite flat if and only if it is finite and
      $L(y)$ is constant as $y$ ranges over the closed points of $Y$.

\item $\phi$ is finite étale if and only if it is finite and
      there is an integer $r\geq 1$ such that, for every closed
      point $y\in Y$, $\#\phi^{-1}(y)=r$ and $e_x(y)=1$ for every $x\in\phi^{-1}(y)$.
\end{enumerate}
\end{proposition}

\begin{remark}\label{etale is covering}
Finite étale morphism can be viewed as an algebraic analogue of a finite covering map in topology, and finite flat morphism is an analogue of a finite covering map where branching is allowed. 
\end{remark}

If $\phi: X \longrightarrow Y$ is a generically finite morphism, it is a standard result in algebraic geometry(cf. \cite{Milne1980}, Chapter I) that there is a nonempty open subset $U$ of $Y$ such that $\phi: \phi^{-1}(U) \longrightarrow U$ is finite(resp. finite flat and finite étale). We therefore have the following definition.  

\begin{definition}\label{finite,flat,etale locus}
Let $\phi:X \longrightarrow Y$ be a generically finite morphism. The finite(resp. finite flat and finite étale) locus of $\phi$ is the maximal open subset $U$ of $Y$ such that $\phi: \phi^{-1}(U) \longrightarrow U$ is finite(resp. finite flat and finite étale). The open set $U$ is denoted by $U_{fin}$(resp. $U_{flat}$ and $U_{\mathrm{\acute et}}$). 
\end{definition}

We now state an important theorem in algebraic geometry, called Miracle Flatness(cf. \cite{Eisenbud1995}\S 18,\cite{Matsumura1986}\S 23), which establishes the equivalence between finiteness and finite flatness under certain assumptions. We begin with the following definition.  

\begin{definition}\label{Cohen--Macaulay}
$X=\operatorname{Spec}(A)$ is \emph{Cohen--Macaulay} if for every prime ideal
$\mathfrak p\subseteq A$, $\operatorname{depth}(A_{\mathfrak p})=\dim(A_{\mathfrak p}).$ Here $\dim(A_{\mathfrak p})$ is the Krull dimension of the local ring $A_{\mathfrak p}$, and $\operatorname{depth}(A_{\mathfrak p})$ is the maximum length of an $A_{\mathfrak p}$-regular sequence contained in
its maximal ideal $\mathfrak pA_{\mathfrak p}$. 
\end{definition}

\begin{theorem}[Miracle Flatness]\label{miracle flatness}
Let $\phi:X \longrightarrow Y$ be a finite morphism. Assume that $Y$ is smooth, $X$ is Cohen--Macaulay, and $\operatorname{dim}(X)=\operatorname{dim}(Y)$. Then $\phi$ is flat. In particular, $\phi$ is finite flat.
\end{theorem}

\begin{remark}
    Any hypersurface in $\C^n$ is Cohen--Macaulay. This fact will be useful in proving lemma \ref{flat locus}. 
\end{remark}

Finally, we introduce the definition of the intersection multiplicity for two curves in a surface $X=\operatorname{Spec}(A)$. Let $X$ be an algebraic surface. Let $C_1, C_2 \subset X$ be two curves with no 0-dimensional components. We assume that $C_1,C_2$ have no common component so that all the intersection points are isolated. Then we have the following definition. 

\begin{definition}\label{interseciton multiplicity}
For any $p \in C_1 \cap C_2$, let $I_{C_1,p}, I_{C_2,p} \subset O_{X,p}$ be the associated ideals of $C_1,C_2$ at $p$. Then the intersection multiplicity of $C_1,C_2$ at $p$ is defined to be 
$$
i_p(C_1,C_2;X)=\operatorname{length}_{O_{X,p}}\frac{O_{X,p}}{I_{C_1,p}+I_{C_2,p}}.
$$
\end{definition}

\begin{remark}
Since finite étale morphisms behave like finite covering maps, the local intersection multiplicity of two curves is preserved under finite étale morphism. See subsection \ref{algorithm and generalization} for more clarification and application. 
\end{remark}

\bigskip
\bigskip

\subsection{properties of boundary restriction map}\label{property of phi for whitehead}
As mentioned in the introduction, the key idea of our approach is to consider the boundary restriction map $\phi$ from the canonical component of the given character variety $X_0(M)$ onto the character variety of the corresponding boundary torus $X(T)$. Then, the intersection pattern can be analyzed through the push-forward and pull-back of $\phi$, provided that we have a thorough understanding of $\phi$ itself, e.g. knowing its map degree, finite(flat, étale)locus, singular fibers etc. In this subsection, we take the Whitehead link as an illustrative example to demonstrate the definition of the map $\phi$ and to show how various properties of $\phi$ can be computed and proved.

Let $M$ be the complement of a neighborhood of the Whitehead link in $\mathbb{S}^3$. We use the following presentation for the fundamental group of $M$
\begin{equation}\label{fundamentral group presentation for whitehead}
\pi_1(M)=\left\langle a,b\ \middle |~ aw=wa \right\rangle,\qquad
w=bab^{-1}a^{-1}b^{-1}ab.
\end{equation}
For a representation $\rho:\pi_1(M)\longrightarrow \operatorname{SL}_2(\C)$, write $A=\rho(a),~B=\rho(b)$ and $W=\rho(w)$. We introduce the trace coordinates 
\begin{equation}\label{trace coordinates x,y,z}
x=\tr(A),\qquad y=\tr(B),\qquad z=\tr(AB).
\end{equation}

From \cite{Landes}, the canonical component of the Whitehead link character variety is the irreducible affine surface $X=X_0(M)=V(F)\subset \A^3_{x,y,z},$
where 
\begin{equation}\label{defining equation for X_0(M) Whitehead}
F(x,y,z)=z^3-xyz^2+(x^2+y^2-2)z-xy.
\end{equation}
Define $\kappa=x^2+y^2+z^2-xyz-2.$ Then $F=z\kappa-xy.$ Hence, in both the coordinate ring and the function field of $X$, we have the relation $z\kappa=xy.$ Let $K=\C(X)$ denote the function field of $X$.

We fix one of the boundary tori $T$ with peripheral basis $\mu=a,~\lambda=wa^{-1}.$ The $\operatorname{SL}_2(\C)$-character variety of the boundary
torus is $Y=X(T)=V(G)\subset \A^3_{m,\ell,n},$ where 
\begin{equation}\label{defining equation for X(T)}
G=m^2+\ell^2+n^2-m\ell n-4.
\end{equation}
with the boundary trace coordinates given by
\begin{equation}\label{trace coordinates m,l,n}
m=\tr(\rho(\mu)),\qquad \ell=\tr(\rho(\lambda)),\qquad n=\tr(\rho(\mu\lambda)).
\end{equation}
Since $\mu=a,~\lambda=wa^{-1},~\mu\lambda=w,$ we have
$m=x=\tr(A),~\ell=\tr(WA^{-1}),~\ n=\tr(W)$. The inclusion $T\hookrightarrow M$ induces the boundary restriction morphism $\phi:X\longrightarrow Y$ with $\phi(x,y,z)=\bigl(m,\ell,n\bigr)$. We first want to find an explicit expression for the map $\phi$. To this end, we need the following lemma(see \cite{Mac_and_Reid} $\S 3.4$).

\begin{lemma}{(trace identities)}\label{trace identities} For any $U,V\in\operatorname{SL}_2(\mathbb C)$, we have

    (1). $\operatorname{tr}(UV)=\operatorname{tr}(VU), ~~\operatorname{tr}(U)=\operatorname{tr}(U^{-1})$

    (2). $\operatorname{tr}(UV)+\operatorname{tr}(UV^{-1})=\operatorname{tr}(U)\operatorname{t}(V)$

    (3). $\operatorname{tr}(U_1U_2\cdots U_k)=\operatorname{tr}(U_2\cdots U_kU_1)$

    (4). $\operatorname{tr}([U,V])=\operatorname{tr}^{2}(U)+\operatorname{tr}^{2}(V)+\operatorname{tr}^{2}(UV) -\operatorname{tr}(U)\operatorname{tr}(V)\operatorname{tr}(UV)-2$

    (5). $\operatorname{tr}(U^k) =2S_{k}(\operatorname{tr}(U)) -\operatorname{tr}(U)S_{k-1}(\operatorname{tr}(U))$. In particular, $\operatorname{tr}(U^2) =\operatorname{tr}^2(U) -2$.
\end{lemma}

We now find an explicit expression for the boundary restriction map $\phi$. 
\begin{lemma}\label{trace functions of n and l}
Let $\phi: X=X_0(M) \longrightarrow Y=X(T)$ be the boundary restriction map. Then the coordinate maps $m(x,y,z),\ell(x,y,z),n(x,y,z)$ for $\phi$ can be given as follows 
$$
m=\operatorname{tr}(A)=x,
$$
$$
n=\operatorname{tr}(W)=(x-yz)\kappa-x+2yz,
$$
and 
$$
\ell=\operatorname{tr}(WA^{-1})=\kappa^2-y^2\kappa+2y^2-2.
$$
\end{lemma}

\begin{proof}
We begin with some trace identities. Using \ref{trace identities} (5), we have 
$$
\operatorname{tr}(A^2)=x^2-2,\qquad\operatorname{tr}(B^2)=y^2-2,\qquad 
\operatorname{tr}\bigl((AB)^2\bigr)=z^2-2.
$$
We also use Cayley-Hamilton to derive
$$
A^2=\operatorname{tr}(A)A-I=x \cdot A-I,\qquad B^2=\operatorname{tr}(B)B-I=y \cdot B-I.
$$
Therefore,
$$
\operatorname{tr}(A^2B)=xz-y, \qquad  \operatorname{tr}(AB^2)=yz-x.
$$

The fact that $m=x$ follows easily from $a=\mu$. We then compute $n=\operatorname{tr}(W)$. Write $C=BAB^{-1}A^{-1},~ D=B^{-1}AB.$ Then $W=CD.$ By \ref{trace identities} (3), it is clear that $\operatorname{tr}(D)=\operatorname{tr}(A)=x.$ Using \ref{trace identities} (2), we have $\operatorname{tr}(CD)+\operatorname{tr}(CD^{-1})=\operatorname{tr}(C)\operatorname{tr}(D).$ Thus 
$$
n=x \cdot \operatorname{tr}(C)-\operatorname{tr}(CD^{-1}).
$$ 
By \ref{trace identities} (4), $\operatorname{tr}(C)=x^2+y^2+z^2-xyz-2=\kappa.$ It remains to compute $S:=\operatorname{tr}(CD^{-1}).$ Since $D^{-1}=B^{-1}A^{-1}B,$ we have $CD^{-1}=BAB^{-1}A^{-1}B^{-1}A^{-1}B.$
Using \ref{trace identities} (3) and Cayley-Hamilton with $B^2=\operatorname{tr}(B)B-I$, we have
$$
S=\operatorname{tr}\left(B^2AB^{-1}A^{-1}B^{-1}A^{-1}\right)=y\cdot \operatorname{tr}(BAB^{-1}A^{-1}B^{-1}A^{-1}) -\operatorname{tr}(AB^{-1}A^{-1}B^{-1}A^{-1}).
$$
Then using \ref{trace identities} (2) and (1), we have $\operatorname{tr}(AB^{-1}A^{-1}B^{-1}A^{-1})=\operatorname{tr}(AB^2)=yz-x$. Separate $BAB^{-1}A^{-1}B^{-1}A^{-1}$ into the product of $BAB^{-1}A^{-1}$ and $B^{-1}A^{-1}$. Then applying \ref{trace identities} (2), we can derive $\operatorname{tr}(BAB^{-1}A^{-1}B^{-1}A^{-1})=\kappa z-z$. Collecting all the identities gives 
$$
n=(x-yz)\kappa-x+2yz. 
$$ 

Now, we want to compute $\ell=\operatorname{tr}(WA^{-1})$. Notice that $WA^{-1}=BAB^{-1}A^{-1}B^{-1}ABA^{-1}$. Separate $WA^{-1}$ into the product of $BAB^{-1}A^{-1}$ and $B^{-1}ABA^{-1}$. Then we can apply \ref{trace identities} (2) and (3) to derive
$$
\operatorname{tr}(WA^{-1})=\operatorname{tr}(BAB^{-1}A^{-1})\operatorname{tr}(B^{-1}ABA^{-1}) - \operatorname{tr}(B^{2}AB^{-2}A^{-1})
$$
Each trace term on the right-hand side of the equation above can be evaluated using \ref{trace identities} (4) and the identities $\operatorname{tr}(B^{-1}A)=xy-z,~\operatorname{tr}(B^2)=y^2-2,~\operatorname{tr}(B^2A)=yz-x$. After direct computations, we can derive 
$$
\ell=\kappa^2-y^2\kappa+2y^2-2.
$$
This completes the proof.
\end{proof}

\begin{remark}
If the other cusp of the Whitehead link is chosen, then the expression for $\phi$ will essentially be the same, except that the positions for variables $x$ and $y$ are exchanged. This is due to the symmetry of the two cusps in the Whitehead link, which also holds for any 2-bridge link in $\mathbb{S}^3$. Therefore, for 2-bridge links, it does not matter which cusp we pick to define the boundary restriction map. 
\end{remark}

The first property of the boundary restriction map $\phi$ we want to show is the following, which we need lemma \ref{quadratic extension of K over C(x,kappa)} and \ref{C(m,l,n)=C(x,kappa)} for its proof. 

\begin{theorem}\label{restriction morphism is dominant and degree 2}
The boundary restriction morphism $\phi:X_0(M)\longrightarrow X(T)$ is dominant and generically finite of degree $2$.
\end{theorem}

\begin{lemma}\label{quadratic extension of K over C(x,kappa)}
Let $K$ be the function field of $X_{0}(M)$ and view $x,\kappa$ as elements in $K$. Then 

(1). $[K:\C(x,\kappa)]=2.$

(2). $\trdeg_{\C}\C(x,\kappa)=2$. Therefore, $x$ and $\kappa$ are algebraically independent over $\C$.
\end{lemma}

\begin{proof}
    We first determine the degree of the function field $K$ over
$\C(x,\kappa)$. Since $z\kappa=xy$ on $X_0(M)$, we have, in the function field $K$, $y=\frac{\kappa z}{x}.$  Substituting $y=\frac{\kappa z}{x}$ into $\kappa=x^2+y^2+z^2-xyz-2$ and collecting the terms gives
$$
z^2=\frac{x^2(\kappa-x^2+2)}{\kappa^2-x^2\kappa+x^2} \in\C(x,\kappa). 
$$
Moreover, from $y=\frac{\kappa z}{x}.$ It follows that $K=\C(x,\kappa)(z),$ and hence $[K:\C(x,\kappa)]\leq 2.$

We now show that the extension has degree exactly $2$. Consider the involution map $\tau(x,y,z)=(x,-y,-z)$. $X_0(M)$ is defined as the vanishing locus of $F(x,y,z)=z^3-xyz^2+(x^2+y^2-2)z-xy$. Since $F(x,-y,-z)=-F(x,y,z)$, $X_0(M)$ is preserved by $\tau$ and hence $\tau$ induces an automorphism $\tau_{*}$ on $K$. Notice that $\tau_{*}(x)=x$ and $\tau_{*}(\kappa)=\kappa$. Therefore, $\tau_{*}$ fixes $\C(x,\kappa) \subset K$. On the other hand, $\tau_{*}(z)=-z\neq z$. So $z \notin \C(x,\kappa)$ and $[K:\C(x,\kappa)] = 2.$

In particular, $K$ is a finite algebraic extension of $\C(x,\kappa)$. Since $X_{0}(M)$ is an algebraic surface, $\trdeg_{\C}K=2.$ It follows that $\trdeg_{\C}\C(x,\kappa)=2.$ Thus $x$ and $\kappa$ are algebraically independent over $\C$.
\end{proof}

\begin{lemma}\label{C(m,l,n)=C(x,kappa)}
$\C(m,\ell,n)=\C(x,\kappa).$
\end{lemma}

\begin{proof}
from lemma \ref{trace functions of n and l}, we can write $n,\ell \in \C(x,y,z)$ as $n(x,y,z)=(x-yz)\kappa-x+2yz$ and $\ell(x,y,z)=\kappa^2-y^2\kappa+2y^2-2$. We now rewrite $n,\ell $ in terms of $x$ and $\kappa$, when view them as elements in the function field of $X_0(M)$. Notice that along $X_0(M)$, we have $y=\frac{\kappa z}{x},~ z^2=\frac{x^2(\kappa-x^2+2)}{\kappa^2-x^2\kappa+x^2}$. Substituting these identities into $n(x,y,z)$ and $\ell(x,y,z)$, we obtain
\begin{equation}\label{n l in terms of x kappa}
    n=x\,\frac{-\kappa^2+4\kappa-x^2}{\kappa^2-x^2\kappa+x^2},\qquad \ell=\frac{(2-x^2)\kappa^2+2x^2\kappa-2x^2}{\kappa^2-x^2\kappa+x^2}.
\end{equation}
Since $m=x,$ it follows that $m,\ell,n\in \C(x,\kappa).$ Therefore, $\C(m,\ell,n)\subseteq \C(x,\kappa).$

We next prove the reverse inclusion $\C(x,\kappa)\subseteq \C(m,\ell,n).$ Since $m=x,$ we already have $x\in \C(m,\ell,n).$ Set $u=x^2,~t=\frac{n}{x}.$ Since $x=m$, both $u$ and $t$ belong to $\C(m,\ell,n)$. Substituting $u,t$ into $n(x,\kappa)$ of \eqref{n l in terms of x kappa}, we obtain
$$
(t+1)\kappa^2-(tu+4)\kappa+u(t+1)=0.
$$
Similarly, Substituting $u,t$ into $\ell(x,\kappa)$ of \eqref{n l in terms of x kappa}, we obtain
$$
(\ell+u-2)\kappa^2-u(\ell+2)\kappa+u(\ell+2)=0.
$$
Combining the two equations above gives 
$$
\kappa=\frac{u(t+1)(4-u)}{u(t+1)(\ell+2)-(\ell+u-2)(tu+4)}.
$$

The denominator in the equation above is nonzero since its evaluation at the point $(1,0,0) \in X_0(M)$ is not $0$. It follows that $\kappa\in\C(m,\ell,n).$ Since $x=m$, we conclude that $\C(x,\kappa)\subseteq \C(m,\ell,n).$ Combining this with the reverse inclusion proved above gives $\C(m,\ell,n)=\C(x,\kappa),$ as desired.
\end{proof}

We are now ready to prove Theorem \ref{restriction morphism is dominant and degree 2}. 

\begin{proof}[proof of Theorem \ref{restriction morphism is dominant and degree 2}]
We first prove that $\phi:X\longrightarrow Y$ is dominant. The function field generated by the coordinate functions of the image of $\phi$ is $\C(m,\ell,n).$ By the previous section, $\C(m,\ell,n)=\C(x,\kappa).$ Since $x$ and $\kappa$ are algebraically independent, $\trdeg_{\C}\C(m,\ell,n)=2.$ Therefore, the Zariski closure of the image of $\phi$ has dimension
$$
\dim\overline{\phi(X)}=\trdeg_{\C}\C(m,\ell,n)=2.
$$
On the other hand, $\phi(X)\subseteq Y,$ and $Y=X(T)$ is the irreducible affine surface $Y=V(m^2+\ell^2+n^2-m\ell n-4).$ Thus $\overline{\phi(X)}\subseteq Y.$ Since $Y$ is irreducible of dimension $2$, every proper closed
subset of $Y$ has dimension strictly less than $2$. Hence $\overline{\phi(X)}=Y.$ Therefore, $\phi:X\longrightarrow Y \text{ is dominant}.$

Because $\phi$ is dominant, it induces an inclusion of function
fields $\phi^*:\C(Y)\hookrightarrow \C(X).$ Under the present coordinates, $\phi^*\C(Y)=\C(m,\ell,n).$ We have proved that $\C(m,\ell,n)=\C(x,\kappa).$ Therefore,
$$
\deg\phi=\left[\C(X):\phi^*\C(Y)\right]= \left[K:\C(m,\ell,n)\right]=\left[K:\C(x,\kappa)\right]=2.
$$
Thus, $\deg\phi=2.$
\end{proof}

The degree $\deg\phi=[\C(X):\phi^*\C(Y)]$ is the generic degree. From subsection \ref{prelimaries 3.1}, this implies that there exists a nonempty Zariski-open subset $U_{\mathrm{\acute et}}\subseteq Y=X(T)$ such that every point of $U_{\mathrm{\acute et}}$ has exactly two distinct preimages. However, this generic degree alone does not imply that $\phi$ is finite everywhere and special fibers with ramification or positive dimension may be included. Therefore, to obtain the correct intersection pattern, we need to find $U_{fin}$, $U_{flat}$, $U_{\mathrm{\acute et}}$ and the singular fibers of $\phi$, and these are precisely what the remainder of this subsection is devoted to.

We first determine the maximal open subset of the target over which the boundary restriction morphism is finite.

\begin{lemma}\label{finite locus}
$U_{fin}=Y\setminus\left(V_Y(m^2-4)\cup V_Y(\delta)\right)$, where $Y=X(T)$ and $\delta=2-\ell+mn\in \C[m,\ell,n]$.
\end{lemma}

\begin{proof}
We first explain why $V_{Y}(m^2-4)$ must be removed from the finite locus. Since $m=x$, the fiber over a point with $m=2$ is contained in the plane section of $X_0(M)$ with $x=2$. Substituting $x=2$ into $F$, we obtain $F(2,y,z)=(z-y)\bigl(z(z-y)+2\bigr).$ Thus $X_0(M)\cap\{x=2\}=C_{1}\cup C_{2}$, where $C_{1}=V(x-2,y-z)$ and $C_{2}=V(x-2,z(z-y)+2)$. On the other hand, the equation of $Y$ restricted to $m=2$ is $G(2,\ell,n)=(\ell-n)^2.$ Thus $Y\cap\{m=2\}=\{(2,t,t):t\in\C\}.$

On $C_{1}$, we have $y=z$, and hence $\kappa=2.$ Substituting these values into the expressions of $n,\ell$ in lemma \ref{trace functions of n and l} gives $(m,\ell,n)=(2,2,2).$ On $C_{2}$, we have $z(z-y)+2=0$. Let $B=\kappa^2-x^2\kappa+x^2$. It is straightforward to check that on $C_2$, we must have $B \neq 0$. Therefore, when we restrict $\phi$ to $C_2$, $n(x,y,z)$ and $\ell(x,y,z)$ take the form in \eqref{n l in terms of x kappa}. Since $\phi(C_2) \subset Y\cap\{m=2\}$, equating $n =\ell$ in \eqref{n l in terms of x kappa} with $x=2$ gives $\kappa=1$ and $(m,\ell,n)=(2,-2,-2).$ Therefore, 
$$
\phi(X_0(M))\cap\{m=2\}=\{(2,2,2),(2,-2,-2)\}.
$$
Since a finite dominant morphism over an irreducible open subset has closed image containing the generic point, no open neighborhood of such a point can support a finite degree-two restriction of $\phi$. Hence the entire divisor $Y\cap V(m-2)$ must be removed from the finite locus. Similarly, $V_{Y}(m+2)$ must also be removed from the finite locus.

We then explain why $V_Y(\delta)$ must be excluded from the finite locus. The equation $\delta=2-\ell+mn=0$ cuts out a curve in $Y$. Indeed, substituting $\ell=2+mn$ into $m^2+\ell^2+n^2-m\ell n-4=0$ gives $(m+n)^2=0.$ Thus, set-theoretically,
$$
V_Y(\delta)=\left\{(m,\ell,n):n=-m,\quad\ell=2-m^2\right\}.
$$

Now suppose that $\xi=(m,\ell,n)\in V_Y(\delta)$ has a preimage $P=(x,y,z)\in X_{0}(M)$ under $\phi$, i.e. $\phi(x,y,z)=(m,\ell,n)$. From lemma \ref{trace functions of n and l}, the expressions for $m,\ell,n$ in terms of $x,y,z$ under $\phi$ are
$$
m=x, \qquad n=(x-yz)\kappa-x+2yz, \qquad \ell=\kappa^2-y^2\kappa+2y^2-2.
$$
Using these expressions, it is not hard to compute that $\delta B=x^2(4-x^2)$, where $B=\kappa^2-x^2\kappa+x^2$. Since $\delta=0$ along $V_{Y}(\delta)$, a point of $V_Y(\delta)$ can have a preimage only if $m=0,~2,~\text{or}-2.$ Consequently, every point $(m,2-m^2,-m)$ with $m\notin\{0,\pm2\}$ has empty fiber. For the similar reason as above, the entire divisor $V_Y(\delta)$ must be removed from the maximal finite locus.

%Since these points form a dense open subset of the irreducible curve $V_Y(\delta)$, no point of $V_Y(\delta)$ can belong to a target open subset over which $\phi$ is finite and dominant. Indeed, a finite dominant morphism to an irreducible variety is surjective.

We now prove $U=U_{fin}$ where $U=Y\setminus\left(V(m^2-4)\cup V(\delta)\right)$. From the arguments above, it suffices to prove that $\phi:\phi^{-1}(U) \longrightarrow U$ is finite. Let $R$ and $A$ be the coordinate rings of $X$ and $Y$ respectively,i.e. $X=\operatorname{spec}(R), Y=\operatorname{spec}(A)$(Note that we are abusing notations here by identifying an affine scheme with the underlying algebraic variety that corresponds to its closed points). The boundary restriction map $\phi: X \longrightarrow Y$ induces a ring homomorphism $\phi^{*}: A \longrightarrow R$, and hence $R$ can be viewed as an A-module via $ar \mapsto \phi^{*}(a)r$. To simplify the notation, we identify $\phi^{*}(a) \in R$ with $a \in A$. Set $s=(m^2-4)\delta$ and let $A_s,R_s$ be the localizations at $s$ of $A,R$ respectively. Thus $\operatorname{Spec}A_{s}=D_Y\bigl((m^2-4)\delta\bigr)=U$ and $\operatorname{spec}(R_{s})=\operatorname{spec}(R\otimes_A A_{s})=\phi^{-1}(U)$. To prove that $\phi:\phi^{-1}(U) \longrightarrow U$ is finite, it suffices to prove that $R_s$ is a finite $A_s$-module. We emphasize that for the following discussion, $x,y,z,m,\ell,n$ should be viewed as elements in $A_s$ or $R_s$, rather than coordinates in $X$ or $Y$. Moreover, $\phi^{*}(m),\phi^{*}(\ell), \phi^{*}(n)$ are given as in lemma \ref{trace functions of n and l}.

Since $x=m$ already lies in $A_{s}$, the localized coordinate ring $R_{s}$ is generated as an $A_{s}$-algebra by $y$ and $z$, i.e. $R_{s}=A_{s}[y,z]$. Then $R_s$ is a finite $A_s$-module if and only if $y$ and $z$ are integral over $A_s$, which we now prove. We first show that $\kappa \in A_s$. Recall that $\delta B=x^2(4-x^2).$ It can also be computed that $(x+n)B=x(4-x^2)\kappa$. Therefore $\kappa=\frac{x(x+n)}{\delta} \in A_s$, since $\delta$ is invertible in $A_s$.

We now prove that $z$ is integral over $A_s$. Recall that $Bz^2=x^2(\kappa-x^2+2)$ and $\delta B=x^2(4-x^2).$ Since $m^2-4=x^2-4$ is invertible in $A_s$, it can be computed that 
$$
z^2=\frac{\delta(\kappa-x^2+2)}{4-x^2} \in A_s.
$$
Thus $z$ satisfies the monic polynomial $T^2-c\in A_{s}[T]$, and hence $z$ is integral over $A_{s}$. 

We then prove $y$ is integral over $A_s$. Using the fact that $z\kappa = xy$, the formulas for $n$ and $\ell$ can be rewritten as
$$
-xy^2+2yz=n-x\kappa+x,\qquad(2-x^2)y^2+xyz=\ell+(2-x^2)\kappa+2.
$$
We can regard these as two linear equations in the two quantities $y^2$ and $yz$. The coefficient matrix is
$$
\begin{pmatrix}
-x & 2\\
2-x^2 & x
\end{pmatrix},
$$
whose determinant is $x^2-4.$ Since $m^2-4$ is invertible in $A_{s}$, this matrix is invertible after localization. Solving for $y^2$ gives
$$
y^2=\frac{xn+(x^2-4)\kappa+x^2-2\ell-4}{x^2-4} \in A_s.
$$
Thus $y$ also satisfies a quadratic monic polynomial in $A_{s}[T],$ and therefore $y$ is integral over $A_{s}$. This completes the proof.
\end{proof}

We then find the finite flat locus of $\phi$. 

\begin{lemma}\label{flat locus}
$U_{flat}=U_{fin}=Y\setminus\left(V_Y(m^2-4)\cup V_Y(\delta)\right).$
\end{lemma}

\begin{proof}
The target surface is
$Y=V(G),~G=m^2+\ell^2+n^2-m\ell n-4.$ Its partial derivatives are 
$$
G_m=2m-\ell n,\qquad G_\ell=2\ell-mn,\qquad G_n=2n-m\ell.
$$
Thus a singular point must satisfy $2m=\ell n,~2\ell=mn,~2n=m\ell.$ Together with $G=0$, these equations give $\Sing(Y)=\left\{(2,2,2),(2,-2,-2),(-2,2,-2),(-2,-2,2)\right\}.$

All singular points of $Y$ lie on the divisor $m^2=4$. Consequently, $U_{fin}$ is a smooth affine scheme. The source $X$ is a hypersurface, and hence Cohen--Macaulay. Therefore $\phi^{-1}(U_{fin})$ is also Cohen--Macaulay. The restricted morphism $\phi: \phi^{-1}(U_{fin})\longrightarrow U_{fin}$ is finite, the target is smooth, and the source and target have the same pure dimension. Thus, by miracle flatness(Theorem \ref{miracle flatness}), this morphism is flat.
\end{proof}

Finally, we determine the finite \'{e}tale locus of $\phi$. 

\begin{theorem}\label{etale locus of phi}
$U_{\mathrm{\acute et}}=X(T)\setminus\left(E_1\cup E_2\cup E_3\right).$
where $E_1=V(m^2-4), E_2=V(2-\ell+mn)$ and $E_3=V(\ell-m^4+4m^2-2,n-m^3+3m)$.
\end{theorem}

\begin{proof}
Consider the involution $\tau:X_0(M)\longrightarrow X_0(M), ~\tau(x,y,z)=(x,-y,-z).$ The key of the proof is to show that $U_{fin}=U_{flat} \cong \phi^{-1}(U_{fin})/\langle \tau \rangle$, where the isomorphism is induced by $\phi$. Notice that the fibers of $E_1,E_2$ are computed in lemma \ref{finite locus} and they are indeed preserved by $\tau$. Therefore $\tau$ restricts to an involution on $\phi^{-1}(U_{fin})$. 

Let $L=\C(X_0(M)),K=\C(X(T))$ be the function fields of $X_0(M)$ and $X(T)$ respectively. Then $\tau,\phi$ induce automorphisms $\tau^{*}:L \longrightarrow L, ~\phi^{*}:K \longrightarrow L$ respectively. We first show that $K=L^{\langle\tau\rangle}$, where $L^{\langle\tau\rangle}$ is the fixed field of $L$ under $\tau^{*}$ and we identify $K$ with $\phi^{*}(K)$. Since $\kappa=x^2+y^2+z^2-xyz-2,$ we have $\kappa\circ\tau=\kappa.$ The formulas in lemma \ref{trace functions of n and l} also give $m\circ\tau=m,~\ell\circ\tau=\ell,~n\circ\tau=n.$ Therefore, $\phi\circ\tau=\phi.$ At the level of function fields, this implies $\tau^{*}\circ \phi^{*}=\phi^{*}$ and hence $\tau^{*}(a)=a ~\forall a\in K$. Therefore $K \subset L^{\langle\tau\rangle}$. By Theorem \ref{restriction morphism is dominant and degree 2}, $[L:K]=2$ and hence either $L^{\langle\tau\rangle}=K$ or $L^{\langle\tau\rangle}=L$. Clearly $y \notin L^{\langle\tau\rangle}$. So we must have $L^{\langle\tau\rangle}=K$. 

Now we let $X=\phi^{-1}(U_{fin}), Y=U_{fin}$ and let $R,A$ be the coordinate rings of $X,Y$ respectively. The involution $\tau: X \to X$ induces an automorphism $\tau^{*}: R \longrightarrow R$ and $X/\langle\tau\rangle=\operatorname{spec}(R^{\langle\tau\rangle})$(Note that we are using $\tau^{*},\phi^{*}$ to denote the induced maps on both the function fields and the coordinate rings). Since $\phi\circ\tau=\phi$, $\phi:X \longrightarrow Y$ descents to a map $\bar{\phi}: X/\langle\tau\rangle \longrightarrow Y$. We again regard $A$ as a subring of $R$ via $\phi^{*}$ and identify $a \in A$ with $\phi^{*}(a) \in R$. Therefore, to show $X/\langle\tau\rangle \cong Y$ via $\bar{\phi}$, it suffices to show $A=R^{\langle\tau\rangle}$. 

Since $\phi \circ \tau=\phi$, this implies $\tau^{*}\circ \phi^{*}=\phi^{*}$ and $A \subset R^{\langle \tau \rangle}$. Since $\phi$ is a finite morphism when restricts to $X$, $R$ is a finite $A$-module and hence $R^{\langle \tau \rangle} \subset R$ is integral over $A$. Since the singular points of $X(T)$ as described in lemma \ref{flat locus} are excluded from $Y$, $Y$ is a smooth affine scheme. This implies $Y$ is normal and hence $A$ is integrally closed in $\operatorname{Frac}(A)$. From the arguments above, we have $\operatorname{Frac}(A)=K=L^{\langle \tau \rangle}=\operatorname{Frac}(R^{\langle \tau \rangle})$. Let $r\in R^{\langle \tau \rangle}$, then $r$ is integral over $A$ and $r \in L^{\langle \tau \rangle}=K$. Therefore $r \in A$ since $A$ is integrally closed. So $R^{\langle \tau \rangle} \subset A$ and $A=R^{\langle \tau \rangle}$. 

Now, we have $Y \cong X/\langle\tau\rangle$ via the descent of $\phi$, so the ramification locus of $\phi$ on $X$ agrees with the ramification locus of the quotient map by $\tau$. Since the characteristic is zero, the quotient map by $\tau$ is \'{e}tale exactly at points where the involution acts freely. Hence the ramification locus is the fixed locus of $\tau$, which is given by $R_{\mathrm{ram}}=V_X(y,z)$. On $R_{\mathrm{ram}}$, $\kappa=x^2-2,$ and $m=x,~ n=x^3-3x,~ \ell=x^4-4x^2+2$. The branch locus on $Y$ is the closure of the image of $R_{\mathrm{ram}}$ under $\phi$. Thus the branch locus is given by
$$
\mathcal B=V_Y\left(\ell-m^4+4m^2-2,\,n-m^3+3m\right).
$$
This equation, combining with the vanishing locus of $Y=U_{fin}$ given in lemma \ref{finite locus}, gives the desired result. 
\end{proof}

The morphism $\phi$ acts like finite covering map on $\phi^{-1}(U_{\mathrm{\acute et}})$. But in order to obtain the complete intersection pattern, we still need to find the fibers for points outside $U_{\mathrm{\acute et}}$. This is done in the following lemma.

\begin{lemma}[Singular fibers]\label{singular fibers}
Let $E=E_1\cup E_2\cup E_3$ be the non-\'{e}tale locus of $\phi$ on $X(T)$, as described in Theorem \ref{etale locus of phi}. Then 
\begin{equation}\label{pairwise intersection of E_i}
\begin{aligned}
&E_1 \cap E_2=\{ (2,-2,-2),(-2,-2,2)\}=\{s_{+-},s_{--}\},\\ 
&E_1 \cap E_3=\{(2,2,2),(-2,2,-2)\}=\{s_{++},s_{-+}\},\\
&E_2 \cap E_3=\{(0,2,0)\}=\{e_0\}.
\end{aligned}
\end{equation}
Let $\widetilde{\Sigma}$ be the collection of these five intersection points. Then the corresponding fibers of $\widetilde{\Sigma}$ are 
\begin{equation}\label{fiber of five singular points}
\begin{aligned}
&\phi^{-1}(s_{++})=R_{++}:=\{(2,y,y):y\in\mathbb C\},\\
&\phi^{-1}(s_{+-})=I_{+-}:=\left\{(2,y,z):(z-y)^2+y(z-y)+2=0\right\},\\
&\phi^{-1}(s_{-+})=R_{-+}:=\{(-2,y,-y):y\in\mathbb C\},\\
&\phi^{-1}(s_{--})=I_{--}:=\left\{(-2,y,z):(y+z)^2-y(y+z)+2=0\right\},\\
&\phi^{-1}(e_0)=J_0:=\{(0,y,0): y\in \C \}.
\end{aligned}
\end{equation}
The other fibers of $E$ can be described as follows: $\phi^{-1}(t)=\varnothing$ for any $t \in (E_1\cup E_2)\setminus \widetilde{\Sigma}$; $\phi^{-1}(t)$ consists of a single point for any $t \in E_3 \setminus \widetilde{\Sigma}$ and $\phi^{-1}(E_3)=R_{++} \cup R_{-+}\cup J_0 \cup D$ where $D=\{(x,0,0): x\in \C\}$.  
\end{lemma}

\begin{proof}
It can be checked by straightforward substitution that 
\begin{equation*}
    \begin{aligned}
        & E_1=E_1^{+} \cup E_2^{-},~~  E_1^{+}=\{(2,\ell,\ell):\ell \in \C\},~ E_1^{-}=\{(-2,\ell,-\ell):\ell \in \C\},\\
        & E_2=\{(m,2-m^2,-m): m \in \C \}, \\
        & E_3=\{(m,m^4-4m^2+2,m^3-3m): m \in \C\}.
    \end{aligned}
\end{equation*}
Then the pairwise intersection between $E_i$ as shown in \eqref{pairwise intersection of E_i} can be computed directly. The fibers of the four points in $\widetilde{\Sigma}\setminus\{e_0\}$ are already checked in lemma \ref{finite locus}. The fiber of $e_0$ can be computed using lemma \ref{trace functions of n and l} and the defining equation $F=0$ of $X_0(M)$. The fact that $\phi^{-1}(t)=\varnothing$ for any $t \in (E_1\cup E_2)\setminus \widetilde{\Sigma}$ is also proved in lemma \ref{finite locus}. Finally, from Theorem \ref{etale locus of phi}, any $t \in E_3 \setminus \widetilde{\Sigma}$ is a branch point of $\phi$ and has a single point as its preimage. Moreover, since $m=x$, it is straightforward to check $\phi(D) \subset E_3$. The result then follows from the fact that $E_3\setminus \widetilde{\Sigma} $ and $\phi^{-1}(E_3\setminus \widetilde{\Sigma})$ are in 1-1 correspondence. 
\end{proof}

We end this subsection with the following observation, which classifies the reducibility of characters in the singular fibers. 

\begin{lemma}[Reducibility of singular fibers]\label{reducibility of fibers}
Using the notations in lemma \ref{singular fibers}, every character in $R_{++}$ and $R_{-+}$ is reducible, whereas every character in $I_{+-}$, $I_{--}$ and $\phi^{-1}(E_3 \setminus \widetilde{\Sigma})$ is irreducible. Moreover, $J_0$ contains two reducible characters and the rest are all irreducible. 
\end{lemma}

\begin{proof}
A character $x \in X_0(M)$ is reducible if and only if $\kappa=\tr[a,b]=2$, or
equivalently $\mathcal R(x,y,z):=x^2+y^2+z^2-xyz-4=0$. 
On the section $x=2$, $\mathcal R(2,y,z)=(z-y)^2=g^2.$
Thus every character in $R_{++}$ is reducible. On $I_{+-}$, the
equation $g^2+yg+2=0$ implies $g\neq0$, and hence every character in $I_{+-}$ is
irreducible. Similarly, every character in $R_{-+}$ is reducible and every character in $I_{--}$ is irreducible. From lemma \ref{singular fibers}, every point in $\phi^{-1}(E_3\setminus \widetilde{\Sigma})$ takes the form $(x,0,0)$ for $x\neq \pm2,0$. Such a point has $\mathcal{R}(x,0,0)\neq0$ and hence it is irreducible. Finally, $(0,y,0)$ reducible if and only if $y=\pm2$. Therefore, $J_0$ has exactly two reducible characters. 
\end{proof}

\begin{remark}
It can be checked directly that $\phi^{-1}(U_{\mathrm{\acute et}})$ contains both reducible and irreducible characters. 
\end{remark}

\bigskip
\bigskip

\subsection{Algorithm and generalization}\label{algorithm and generalization}
We have computed the degree of $\phi:X_0(M) \longrightarrow X(T)$(Theorem \ref{restriction morphism is dominant and degree 2}), found the various loci of $\phi$(lemma \ref{finite locus}, \ref{flat locus} and Theorem \ref{etale locus of phi}) and determined its singular fibers(lemma \ref{singular fibers}). Now, we put all these results together to decide how $C_{p/q}, C_{r/s}$(cf. \S 2.3) intersect in $X_0(M)$. To keep our discussion consistent throughout, we assume the Dehn fillings are performed on the chosen boundary component $T$. Then the rule of thumb to compute the intersection pattern is the following: we first compute the \textit{possible candidates} of intersection points of $\phi(C_{p/q})$ and $\phi(C_{r/s})$ in $X(T)$. If a candidate lies in the finite \'{e}tale locus of $\phi$, then it corresponds to two distinct intersection points in $X_0(M)$, each with intersection multiplicity one. If a candidate lies in the branch locus, then it corresponds to one intersection point in $X_0(M)$ with intersection multiplicity two. For other candidates, we use lemma \ref{singular fibers} and lemma \ref{singular-fibers-intersection} below to pull back their corresponding intersection pattern in $X_0(M)$.   

Now the remaining question is how to define \textit{possible candidates}. The key idea is the following observation. Let $\chi \in X_0(M)$ and $\rho$ a representative of $\chi$. Let $A=\rho(\mu),B=\rho(\lambda)$. Then $\rho|_{\pi_{1}(T)}$ is abelian and hence reducible. Therefore, after proper conjugation, we can assume $A,B$ are both upper-triangular with diagonal entries given by $(M,M^{-1})$ and $(L,L^{-1})$ respectively, where $(M,M^{-1}), (L,L^{-1})$ are eigenvalues of $A,B$. Then the diagonal entries of $\rho(\mu^{p}\lambda^{q})=A^p B^q$ are given by $(M^pL^q,M^{-p}L^{-q})$. If $\chi \in C_{p/q}\cap C_{r/s}$, then we must have $M^p L^q=M^r L^s=1$. So we define $\mathcal{S}_{p/q,r/s}:=\{(M,L)\in (\C^{*})^2 \mid M^p L^q=M^r L^s=1\}$. Given the eigenvalues of $\rho(\mu),\rho(\lambda)$, the corresponding character in $X(T)$ can be easily obtained from the quotient map $\epsilon:(\mathbb C^*)^2\longrightarrow X(T),$ given in trace coordinates by
\begin{equation}\label{map epsilon}
    \epsilon(M,L)=\left(M+M^{-1},L+L^{-1},ML+M^{-1}L^{-1}\right).
\end{equation}
Thus $X(T)$ is identified with the quotient of $(\mathbb C^*)^2$ by the Weyl involution $\iota(M,L)=(M^{-1},L^{-1})$. We now define the possible candidate set as follows. 

\begin{definition}\label{def of B_gamma,delta}
The possible candidate set of intersection points in $X(T)$ is defined by $\mathcal{B}_{p/q,r/s}:=\epsilon(\mathcal{S}_{p/q,r/s})$, where $\mathcal{S}_{p/q,r/s}:=\{(M,L)\in (\C^{*})^2 \mid M^p L^q=M^r L^s=1\}$ and $\epsilon$ is given by \eqref{map epsilon}. 
\end{definition}

A point $b \in \mathcal{B}_{p/q,r/s}$ is called a "possible" candidate because it does not necessarily correspond to any intersection of $C_{p/q}$ and $C_{r/s}$. One possibility is that the fiber of $b$ may be empty, as described in lemma \ref{singular fibers}. Then, $b$ is not even the restriction of a character in $X_0(M)$. Even if the fiber of $b$ is not empty, it does not necessarily correspond to an intersection either, as demonstrated in lemma \ref{singular-fibers-intersection} (3) below. This is because $M^pL^q=1$ can only guarantee that  $\rho(\mu^p\lambda^q)$ is an upper-triangular matrix with diagonal entries equal to 1. But in order for the fiber to lie in $C_{p/q}$, we need $\rho(\mu^p\lambda^q)=I$ so that the representation really factors through the Dehn filling slope. Fortunately, when a possible candidate $b$ does not lie in the singular locus $\Sigma=\{ (2,2,2),(2,-2,-2),(-2,2,-2),(-2,-2,2)\}$, its fiber, if not empty, will indeed correspond to some intersections of $C_{p/q}\cap C_{r/s}$, as demonstrated by Proposition \ref{correspondence of intersection and fiber} below. To prove it, we first need the following lemma.  

\begin{lemma}\label{simultaneous diagonal}
Suppose that $\rho:\pi_1(T)\longrightarrow \mathrm{SL}_2(\mathbb{C})$ has character $b\in X(T)\setminus \Sigma$. Then with proper conjugation, we have 
\begin{equation}\label{eigenvalues of mu lambda}
\rho(\mu)=
\begin{pmatrix}
M&0\\
0&M^{-1}
\end{pmatrix},
\qquad 
\rho(\lambda)=
\begin{pmatrix}
L&0\\
0&L^{-1}
\end{pmatrix}.
\end{equation}
\end{lemma}

\begin{proof}
Let $b=(m,\ell,n)\in X(T)\setminus \Sigma,$ and let $\rho:\pi_1(T)\longrightarrow \mathrm{SL}_2(\mathbb{C})$ be a representation with character $b$. Because $\pi_1(T)$ is abelian, $\rho(\pi_1(T))$ is is an abelian subgroup of $\mathrm{SL}_2(\mathbb{C})$. Up to conjugation, this means $\rho(\pi_1(T))$ is a subgroup of either $D$ or $S \subset \mathrm{SL}_2(\mathbb{C})$, where $D$ is the group of diagonal matrices and $S$ is the group of upper triangular unipotent matrices and their negatives. If $\rho(\pi_1(T)) \subset S$, then $b \in \Sigma$ and this contradicts the fact that $b \notin \Sigma$. So we must have $\rho(\pi_1(T)) \subset D$. Therefore, with proper conjugation, we can assume $\rho(\mu),\rho(\lambda)$ take the forms as in \eqref{eigenvalues of mu lambda}. 
\end{proof}

\begin{corollary}\label{equality on M,L imply equality on mu,lambda}
Suppose that $M^pL^q=1$ and $b=\epsilon((M,L)) \in X(T)\setminus\Sigma$. If $\phi^{-1}(b)$ is nonempty, then for any $\chi \in \phi^{-1}(b)$ and any representative $\rho$ of $\chi$, we have $\rho(\mu^p \lambda^q)=I$. In particular, $\chi \in C_{p/q}$ for any $\chi \in \phi^{-1}(b)$. 
\end{corollary}

\begin{proof}
Suppose that $b=\epsilon(M,L)\notin \Sigma$. Let $\chi \in\phi^{-1}(b)$ and $\rho:\pi_{1}(M) \longrightarrow \operatorname{SL}(2,\C)$ a representative of $\chi$. Let $\rho_{T}$ be the restriction of $\rho$ to $\pi_1(T)$. Then the character of $\rho_T$ is $b$. From lemma \ref{simultaneous diagonal}, there exists $g\in \pi_1(T)$ such that $g\rho_T(\mu)g^{-1},~g\rho_T(\lambda)g^{-1}$ are both diagonal with entries $(M,M^{-1}),(L,L^{-1})$. Hence
$$
g\rho_{T}(\mu^p\lambda^q)g^{-1}=(g\rho_T(\mu)g^{-1})^{p}(g\rho_T(\lambda) g^{-1})^{q}=\begin{pmatrix}
M^pL^q&0\\
0&M^{-p}L^{-q}
\end{pmatrix}
=I
$$
Therefore, $\rho(\mu^p\lambda^q)=\rho_{T}(\mu^p\lambda^q)=I$. 
\end{proof}

We now summarize all the discussion above into the following important proposition, which connects the intersection $C_{p/q}\cap C_{r/s}$ with the fibers of the possible candidates. 

\begin{proposition}\label{correspondence of intersection and fiber}
Let $\gamma=\frac{p}{q}$ be a slope. Then, every point of $C_\gamma$ lies over(via $\phi$) a point of $\mathcal B_{\gamma}=\epsilon\left(V(M^pL^q-1)\right)$. Conversely, let $b\in\mathcal B_{\gamma}.$ If $b \notin \Sigma$, then $\phi^{-1}(b)\subset C_\gamma$. In particular, let $\gamma=\frac pq,\delta=\frac rs$ be distinct primitive slopes. Then we have
\begin{equation}
\begin{cases}
& C_\gamma\cap C_\delta\subset\phi^{-1}(\mathcal B_{\gamma,\delta}), \\
& \phi^{-1}(\mathcal B_{\gamma,\delta}\setminus \Sigma)
\subset C_\gamma\cap C_\delta.
\end{cases}
\end{equation}
\end{proposition}

Since none of the point in $\Sigma$ lies in the finite(flat) locus $U_{fin}=U_{flat}$ of $\phi$, the pullback pattern of points in $\mathcal B_{\gamma,\delta}\cap U_{fin}$ is easy to describe: each point $y$ in the \'{e}tale locus $U_{\mathrm{\acute et}}\subset X(T)$ will be lifted to two distinct points in $C_\gamma\cap C_\delta\cap \phi^{-1}(U_{fin})\subset X_0(M)$, and each $y$ in the branch locus $U_{fin} \setminus U_{\mathrm{\acute et}}$ will be lifted to a single point in $C_\gamma\cap C_\delta\cap \phi^{-1}(U_{fin})$. For other points in $\mathcal B_{\gamma,\delta}$, we use lemma \ref{singular fibers} and the following lemma to pull back the intersection pattern of $C_\gamma$ and $C_\delta$ in $X_0(M)$. 

\begin{lemma}\label{singular-fibers-intersection}
Let $\frac pq, \frac rs$ be distinct primitive slopes. Then 

\begin{enumerate}
\item $R_{++}\subset C_{p/q}\cap C_{r/s}$ for every pair of slopes $p/q$ and $r/s$.

\item If $p$ and $r$ are both even, then $R_{-+}\subset C_{p/q}\cap C_{r/s}.$ Otherwise, $R_{-+} \cap C_{p/q}\cap C_{r/s}=\varnothing$. 

\item $C_{p/q}\cap C_{r/s}\cap I_{+-}=\varnothing.$ 

\item $C_{p/q}\cap C_{r/s}\cap I_{--}=\varnothing.$
 
\end{enumerate}
\end{lemma}

\begin{proof}
By lemma \ref{reducibility of fibers}, every character in $R_{++}$ is reducible and hence has a diagonal representative of the form
$$
\rho(a)=I,\qquad\rho(b)=
\begin{pmatrix}
u&0\\
0&u^{-1}
\end{pmatrix},
\qquad u+u^{-1}=y.
$$
where $a,b$ are the two generators of $\pi_1(M)$. Substituting $\rho(a)=I$ into $w=bab^{-1}a^{-1}b^{-1}ab$ gives $\rho(w)=I.$ Since $\lambda=wa^{-1}$, it follows that $\rho(\mu)=\rho(\lambda)=I.$ Therefore $\rho(\mu^p\lambda^q)=I$ and hence $R_{++} \subset C_{p/q}$. This holds for for every $p,q$, therefore $R_{++} \subset C_{p/q}\cap C_{r/s}$. 

On $R_{-+}$, a diagonal representative may be chosen so that $\rho(\mu)=-I,~\rho(\lambda)=I.$ It follows that $\rho(\mu^p\lambda^q)=(-I)^p.$ Therefore
$$
C_{p/q}\cap R_{-+}=
\begin{cases}
R_{-+},&p\equiv0\pmod2,\\
\varnothing,&p\equiv1\pmod2.
\end{cases}
$$
In particular, if $p$ and $r$ are both even, then $R_{-+}\subset C_{p/q}\cap C_{r/s}.$ Otherwise, $R_{-+} \cap C_{p/q}\cap C_{r/s}=\varnothing$.

We now consider $I_{+-}$, where $x=2$ and $g^2+yg+2=0$ for $g=z-y \neq 0$. Let $\chi \in I_{+-}$ be a point on this curve and let $\rho:\pi_1(M) \longrightarrow \operatorname{SL}_2(\C)$ be a representation whose character equals $\chi$. By lemma \ref{reducibility of fibers}, $\rho$ must be an irreducible representation. 

Now, let $A=\rho(a), B=\rho(b)$ where $a,b$ are the two generators of $\pi_1(M)$. Since $\operatorname{tr}(A)=x=2$ and $\operatorname{det}(A)=1$, the characteristic polynomial of $A$ is $(t-1)^2$ and hence $A$ is unipotent. If $A=I$, then $z=\operatorname{tr}(AB)=\operatorname{tr}(B)=y$ and $g=0$. This contradicts the fact that $g \neq 0$ on $I_{+-}$, so $A \neq I$. Every nontrivial unipotent matrix in $\operatorname{SL}(2,\C)$ is conjugate to $I+N$ where 
$$
N=
\begin{pmatrix}
0&1\\
0&0
\end{pmatrix}.
$$
We may thus assume $A=I+N$. Let $P_c=I+cN$ for any $c \in \C$. Notice that $P_c AP_{c}^{-1}=A$. It can be shown that after proper conjugation with some $P_c$, we can assume 
$$
B=
\begin{pmatrix}
0&-1/g\\
g&y
\end{pmatrix}.
$$
A direct substitution into the word $w$, followed by the use of $g^2+yg+2=0$, gives
$$
\rho(\lambda)=\rho(wa^{-1})=
\begin{pmatrix}
-1&yz\\
0&-1
\end{pmatrix}
=-\bigl(I-yzN\bigr).
$$
Since $\rho(\mu)=\rho(a)=A=I+N$, we have $\rho(\mu)^p=I+pN$ and $\rho(\lambda)^q=(-1)^q\bigl(I-qyzN\bigr).$ Since $\mu$ and $\lambda$ commute, $\rho(\mu^p\lambda^q)=(-1)^q\bigl(I+(p-qyz)N\bigr).$ Thus $\rho(\mu^p\lambda^q)=I$ if and only if $q$ is even and $p-qyz=0.$

Let $\chi \in I_{+-} \cap C_{p/q}$. If $q=0$, the equation $p-qyz=0$ would imply $p=0$, contradicting the primitivity of $(p,q)$. So $q\neq0$ and $yz=\frac pq$. Suppose now that $\chi \in I_{+-}\cap C_{p/q}\cap C_{r/s}$, then $yz=\frac pq=\frac rs$. It follows that distinct $C_{p/q}$ and $C_{r/s}$ have no common point in $I_{+-}$. Similar arguments also hold for $I_{--}$. 
\end{proof}

Now, it remains to address the issue of intersection multiplicity(cf. \ref{interseciton multiplicity}). By Proposition \ref{correspondence of intersection and fiber}, we have $\phi(C_{\gamma})\setminus \Sigma=\mathcal{B}_{\gamma}\setminus \Sigma$ and $C_{\gamma}\cap\phi^{-1}(X(T)\setminus \Sigma)=\phi^{-1}(\mathcal{B}_{\gamma}\setminus \Sigma)$. In particular, since $\mathcal{B}_{\gamma}$ is a curve with no 0-dimensional component, $C_{\gamma}\cap \phi^{-1}(U_{fin})=\phi^{-1}(\mathcal{B}_{\gamma}\cap U_{fin})$ is of pure dimension 1, if not empty. Moreover, for $\delta\neq \gamma$, since $\mathcal{B}_{\gamma}$ and $\mathcal{B}_{\delta}$ have no common 1-dimensional component, $C_{\gamma}\cap \phi^{-1}(U_{fin})$ and $C_{\delta}\cap \phi^{-1}(U_{fin})$ can not have any common 1-dimensional component. Therefore, for every $x \in C_{\gamma} \cap C_{\delta} \cap \phi^{-1}(U_{fin})$, the intersection multiplicity of $C_{\gamma}$ and $C_{\delta}$ at $x$ is well-defined. We have the following lemma to characterize the intersection multiplicity.  

\begin{lemma}\label{intersection multiplicity computation}
$i_x(C_{\gamma},C_{\delta};X_0(M))=1$ if $x \in \phi^{-1}(U_{\mathrm{\acute et}})$ and $i_x(C_{\gamma},C_{\delta};X_0(M))=2$ if $x \in \phi^{-1}(U_{fin}) \setminus \phi^{-1}(U_{\mathrm{\acute et}})$. 
\end{lemma}

\begin{proof}
Let $x \in C_{\gamma} \cap C_{\delta} \cap \phi^{-1}(U_{fin})$ and $y=\phi(x)$. It is a basic result in algebraic geometry that $i_x(C_{\gamma},C_{\delta};X_0(M))=e_x(y) \cdot i_y(\mathcal{B}_{\gamma},\mathcal{B}_{\delta};X(T))$. We have $e_x(y)=1$ if $y\in U_{\mathrm{\acute et}}$ and $e_x(y)=2$ if $y \in U_{fin} \setminus U_{\mathrm{\acute et}}$. Therefore it suffices to check that $i_y(\mathcal{B}_{\gamma},\mathcal{B}_{\delta};X(T))=1$ for any $y\in \mathcal{B}_{\gamma} \cap \mathcal{B}_{\delta}\setminus \Sigma$. Let $D_{\gamma}=\{ (M,L)\in (\C^{*})^2 \mid M^pL^q=1\}=V(M^pL^q-1)$ and $D_{\delta}=V(M^rL^s-1)$. Then $\mathcal{B}_{\gamma} \cap \mathcal{B}_{\delta}\setminus \Sigma= \epsilon(D_{\gamma}\cap D_{\delta}\setminus \epsilon^{-1}(\Sigma))$. It is straightforward to check that $D_\gamma$ and $D_\delta$ intersects transversely in $\C^2$ for $\delta\neq \gamma$. Since $\epsilon: (\C^{*})^2\setminus \epsilon^{-1}(\Sigma) \longrightarrow X(T)\setminus \Sigma$ is \'{e}tale, the intersection multiplicity is preserved in $X(T)$. Therefore, $i_y(\mathcal{B}_{\gamma},\mathcal{B}_{\delta};X(T))=1$ for any $y\in \mathcal{B}_{\gamma} \cap \mathcal{B}_{\delta}\setminus \Sigma$. This completes the proof. 
\end{proof}

We now have all the preliminaries we need. We can now put everything together to obtain the core of this paper, namely the following algorithm that computes the intersection pattern of $C_\gamma$ and $C_{\delta}$ in the canonical component of the Whitehead link complement, once two distinct slopes $\gamma,\delta$ are taken as input. 

\begin{theorem}[Algorithm for computing intersection pattern]\label{Algorithm for computing intersection pattern}
Let $M$ be the Whitehead link complement in $\mathbb{S}^3$ and $X_0(M)$ the canonical surface of the character variety of $M$. Fix a boundary component $T$ of $M$ and perform two Dehn surgeries with distinct slopes $\gamma=\frac{p}{q},~\delta=\frac{r}{s}$ on $T$ to obtain two generalized curves $C_\gamma, C_\delta \subset X_0(M)$. Let $\phi:X_0(M) \longrightarrow X(T)$ be the boundary restriction map. Then the intersection pattern $C_\gamma \cap C_\delta$ and the corresponding intersection multiplicity can be computed using the following steps. 
\begin{enumerate}
    \item Compute $\mathcal{S}_{\gamma,\delta}=\{(M,L)\in (\C^{*})^2 \mid M^p L^q=M^r L^s=1\}$;
    \item Take Weyl involution and compute $\mathcal B_{\gamma,\delta}=\epsilon(\mathcal{S}_{\gamma,\delta})\subset X(T)$;
    \item For each $b\in \mathcal B_{\gamma,\delta}$, determine the fiber $\phi^{-1}(b)$ and the contribution of the fiber to $C_\gamma \cap C_\delta$, based on the fiber classification of $\phi$.
\end{enumerate}
The fiber classification table for the Whitehead link complement $M$ is summarized as follows. 

\begin{table}[H]
\centering
\begin{tabular}{c|c|c|c}
Location of $b$ &
Fiber of $\phi$ &
Contribution to $C_\gamma\cap C_\delta$ &
Multiplicity \\ \hline
$U_{\mathrm{\acute et}}$ &
two points &
two isolated points &
$1+1$ \\

$U_{fin}\setminus U_{\mathrm{\acute et}}$ &
one point &
one isolated point &
$2$ \\

$s_{++}=(2,2,2)$ &
one-dimensional &
common fiber component &
non-proper \\

$s_{+-}=(2,-2,-2)$ &
one-dimensional &
none &
0 \\

$s_{-+}=(-2,2,-2)$ &
one-dimensional &
common fiber component &
non-proper \\

$s_{--}=(-2,-2,2)$ &
one-dimensional &
none &
0 \\

$e_{0}=(0,2,0)$ &
one-dimensional &
common fiber component &
non-proper \\

Other &
empty &
none &
$0$ \\

\end{tabular}
\caption{fiber classification for the Whitehead link complement}
\label{fiber classification table for whitehead}
\end{table}

\end{theorem}

\begin{remark}\label{advantage of the algorithm}
The main advantage of algorithm \ref{Algorithm for computing intersection pattern} is that it breaks down the whole procedure of computing the intersection into two independent stages. The first stage involves the first two steps of the algorithm, which depend only on the slope coefficients $\frac{p}{q},\frac{r}{s}$ and are easy to compute. The second stage is to find the fiber classification table(cf. Table \ref{fiber classification table for whitehead}). This stage relies on a comprehensive understanding of the boundary restriction map $\phi$, including its degree of map, finite(resp. flat and \'{e}tale) locus, singular fibers, etc. All of this information(except for possibly the intersection pattern in the fibers of the 4 singular points in $\Sigma$) is completely determined by the map $\phi$ and is independent of the slope coefficients. Therefore, the table can be computed and stored \textit{a priori}. Once the two stages are completed, the final step of the algorithm is simply tracking the table point by point, which can be easily done with the assistance of a computer. 
\end{remark}

\begin{remark}\label{R_++ contained in all C_gamma}
Note that for any slopes $\gamma,\delta$, $(1,1) \in \mathcal{S_{\gamma,\delta}}$ and hence $s_{++}=(2,2,2)\in \mathcal{B}_{\gamma,\delta}$. Based on table \ref{fiber classification table for whitehead}, this point will contribute a common line component $R_{++}=\{(2,y,y):y\in \C\}$ to $C_\gamma \cap C_{\delta}$. Therefore, $R_{++}$ is contained in all $C_\gamma$ and $C_\gamma$ is a subvariety of dimension 1 in $X_0(M)$(where $M$ is the Whitehead link complement) for any slope $\gamma$. 
\end{remark}

Based on remark \ref{advantage of the algorithm}, algorithm \ref{Algorithm for computing intersection pattern} can be generalized to other 2-bridge links in $\mathbb{S}^3$. We refer the reader to \cite{Zieschang} $\S 12$ for an introduction to 2-bridge knots and links. If $M$ is a 2-bridge link in $\mathbb{S}^3$, its fundamental group also takes the presentation $\pi_1(M)=\langle a,b \mid ~aw=wa\rangle$, where $w$ is a word in $a,b$ determined by the Schubert normal form of $M$. Therefore, the coordinates for $X(M)$ can be chosen again as $(x,y,z)$, as in \eqref{trace coordinates x,y,z}. We can again fix a boundary torus $T$ and assume that the Dehn fillings are all performed on $T$. The defining equation for the character variety $X(T)$ remains the same for all 2-bridge links. Hence, fix a canonical surface $X_0(M)$(cf. Theorem \ref{Thurston's theorem}) of the character variety of $M$, the boundary restriction map $\phi$ can be defined exactly the same way. The explicit expression for $\phi$ can be computed in a manner similar to that of lemma \ref{trace functions of n and l}, using the new word expression $w$. Due to the symmetry of the boundary components for 2-bridge links, changing the boundary torus $T$ will result only in an exchange of the variables $x,y$ in the expression of $\phi$, and therefore it does not matter which torus boundary we fix. The notion of \textit{possible candidates} in $X(T)$ and the idea of push-forward(resp. pull-back) via $\phi$ can be applied exactly in the same way. In particular, Proposition \ref{correspondence of intersection and fiber} holds for any 2-bridge link in $\mathbb{S}^3$. 

The only difference from the Whitehead link case is that the new boundary restriction map $\phi$ will exhibit different behaviors. Therefore, we need to redo subsection \ref{property of phi for whitehead} for this new $\phi$ to compute its map degree, various loci and singular fibers. Once these results are obtained, the rest follow in a manner similar to Theorem \ref{Algorithm for computing intersection pattern}. We summarize the steps in the following theorem. 

\begin{theorem}[Algorithm for any 2-bridge link]\label{Algorithm for computing intersection pattern in general}
Let $M$ be a 2-bridge link in $\mathbb{S}^3$ and $X_0(M)$ a given canonical surface of the character variety of $M$. Fix a boundary component $T$ of $M$ and perform two Dehn surgeries with distinct slopes $\gamma=\frac{p}{q},~\delta=\frac{r}{s}$ on $T$ to obtain two generalized curves $C_\gamma, C_\delta \subset X_0(M)$. Let $\phi:X_0(M) \longrightarrow X(T)$ be the boundary restriction map. Then the intersection pattern $C_\gamma \cap C_\delta$ and the corresponding intersection multiplicity can be computed using the following steps. 
\begin{enumerate}
    \item Compute the relevant properties of $\phi$. This includes finding the map degree, the finite(resp. flat, \'{e}tale) locus, and the singular fibers. A table similar to Table \ref{fiber classification table for whitehead} should be obtained. 
    \item Compute $\mathcal{S}_{\gamma,\delta}=\{(M,L)\in (\C^{*})^2 \mid M^p L^q=M^r L^s=1\}$;
    \item Take Weyl involution and compute $\mathcal B_{\gamma,\delta}=\epsilon(\mathcal{S}_{\gamma,\delta})\subset X(T)$;
    \item For each $b\in \mathcal B_{\gamma,\delta}$, determine the fiber $\phi^{-1}(b)$ and the contribution of the fiber to $C_\gamma \cap C_\delta$, based on the fiber classification of $\phi$ obtained in step (1).
\end{enumerate}
\end{theorem}

\begin{remark}
Once the degree $k$ of the boundary restriction map $\phi$ is computed for a given 2-bridge link $M$, any possible candidate in $U_{\mathrm{\acute et}}\subset X(T)$ will correspond to $k$ distinct intersection points in $X_0(M)$, each with multiplicity $1$, using similar arguments in lemma \ref{intersection multiplicity computation}. However, for candidates in the branch locus, we need to compute the fiber multiplicities of $\phi$ along these branch points to obtain the correct number of intersection points and multiplicities. This requires some additional effort compared to the Whitehead link case, where the degree of $\phi$ is 2 and hence the fiber multiplicity at a ramification point is easily seen to be 2. 
\end{remark}

\subsection{Examples}
In this subsection, we provide a wealth of examples in the case of the Whitehead link to demonstrate the different intersection patterns that could possibly arise. Let $\gamma=\frac{p}{q},~\delta=\frac{r}{s}$ be two primitive slopes. We stick to the notations in the previous subsections and let
$\mathcal{S}_{p/q}=\{(M,L)\in(\C^*)^2:M^pL^q=1\}$, $\mathcal{B}_{\gamma,\delta}=\epsilon(\mathcal{S}_{\gamma,\delta})\subset X(T)$. The various loci of the boundary restriction map $\phi$ are given by: $U_{fin}=U_{flat}=X(T)\setminus(E_1\cup E_2)$ and $U_{\mathrm{\acute et}}=X(T)\setminus\left(E_1\cup E_2\cup E_3\right)$, where $E_1=V(m^2-4), E_2=V(2-\ell+mn)$ and $E_3=V(\ell-m^4+4m^2-2,n-m^3+3m)$. The singular points of $X(T)$ are $\Sigma=\{s_{++},s_{+-},s_{-+},s_{--}\}$ where $s_{++}=(2,2,2),~s_{+-}=(2,-2,-2),~s_{-+}=(-2,2,-2),~s_{--}=(-2,-2,2)$. There is one additional special point in $X(T)$, given by $e_0=(0,2,0)$. 

\subsubsection{Case 1: \'{e}tale intersection point}
Let $\gamma=1/2$ and $\delta=1/5$. The peripheral equations are $ML^2=1,~ML^5=1$. 
Hence $L^3=1,~M=L$ and consequently $\mathcal{S}_{\gamma,\delta}=\{(1,1),(\zeta_3,\zeta_3),(\zeta_3^2,\zeta_3^2)\}$, where $\zeta_3=e^{2\pi i/3}$.

The point $(1,1)$ contributes a common line component $R_{++}$ to $C_\gamma\cap C_{\delta}$, as stated in remark \ref{R_++ contained in all C_gamma}. $(\zeta_3,\zeta_3)\sim(\zeta_3^2,\zeta_3^2)$ determine the same image $(-1,-1,-1) \in U_{\mathrm{\acute et}}$ under the Weyl involution. Since this point lies in $\mathcal{B}_{\gamma,\delta} \cap U_{\mathrm{\acute et}}$, it lifts to two
distinct intersection points in $X_0(M)$. Therefore, aside from $R_{++}$,  $C_{\gamma}\cap C_{\delta}$ has two isolated intersection points, each with intersection multiplicity 1. 

\bigskip
Open sets in Zariski topology are "big" in the sense that they are all dense in the ambient space. In particular, the \'{e}tale locus of $\phi$ is dense in $X(T)$. Therefore, most of the possible candidates in $\mathcal{B}_{\gamma,\delta}$ will lie in $U_{\mathrm{\acute et}}$ and the example above will account for most of the intersection patterns. However, other intersection behaviors are still possible, and the following examples are provided to demonstrate the wide range of possibilities in intersection patterns. 

\subsubsection{Case 2: a point of the branch locus}
Let $\gamma=\frac{1}{1},~\delta=\frac{7}{1}$. The peripheral equations are $ML=1,~M^7L=1$. Let $\zeta_3=e^{2\pi i/3}$ and choose $(M,L)=(\zeta_3,\zeta_3^2)$. Then $(M,L) \in \mathcal{S}_{\gamma,\delta}$ and the corresponding coordinates in $\mathcal{B}_{\gamma,\delta}$ are given by
$$
m=\zeta_3+\zeta_3^{-1}=-1,\qquad \ell=\zeta_3^2+\zeta_3^{-2}=-1,\qquad  n=ML +(ML)^{-1}=2.
$$
Thus $b=(-1,-1,2)\in E_{3}\setminus(E_1 \cup E_2)$. Therefore, the fiber of $b$ consists of a single ramification point and $C_{\gamma}\cap C_{\delta}\cap\phi^{-1}(b)=\{P\}$. This is an example in which the intersection downstairs in $X(T)$ occurs at the branch locus and this intersection point contributes to one geometric intersection point with multiplicity 2 in $X_0(M)$.

\subsubsection{Case 3: a general point of $E_1$}
Let $\gamma=\frac{1}{3},~\delta=\frac{2}{3}.$ The peripheral equations are $ML^3=1,~M^2L^3=1$. Let $\zeta_3=e^{2\pi i/3}$ and choose $(M,L)=(1,\zeta_3)$. Then $(M,L) \in \mathcal{S}_{\gamma,\delta}$ and the corresponding coordinates in $\mathcal{B}_{\gamma,\delta}$ are given by
$$
m=2,\qquad \ell=\zeta_3+\zeta_3^{-1}=-1,\qquad n=\zeta_3+\zeta_3^{-1}=-1.
$$
Thus $b=(2,-1,-1) \in E_1\setminus \Sigma$. Based on the fiber classification table \ref{fiber classification table for whitehead}, its fiber is empty: $\phi^{-1}(b)=\varnothing.$ Consequently, $C_{\gamma}\cap C_{\delta}\cap\phi^{-1}(b)=\varnothing$. This is an example in which the intersection point downstairs in $X(T)$ does not lift to the canonical component.

\subsubsection{Case 4: a general point of $E_2$}
Let $\gamma=\frac{1}{2},~\delta=\frac{7}{2}.$ The peripheral equations are $ML^2=1,~M^7L^2=1$. Choose $M=\zeta_3=e^{2\pi i/3},~L=-M^{-2}=-\zeta_3$. Then $(M,L) \in \mathcal{S}_{\gamma,\delta}$ and the corresponding coordinates in $\mathcal{B}_{\gamma,\delta}$ are given by
$$
m=\zeta_3+\zeta_3^{-1}=-1,\qquad l=(-\zeta_3)+(-\zeta_3)^{-1}=1,\qquad n=ML+(ML)^{-1}=1.
$$
Thus $b=(-1,1,1)\in E_2\setminus \Sigma$. As in the preceding case, the fiber is empty and $C_{\gamma}\cap C_{\delta}\cap\phi^{-1}(b)=\varnothing$.

\subsubsection{Case 5: a point of $E_1\cap E_2$}
Let $\gamma=\frac{1}{2},~\delta=\frac{3}{2}$. The peripheral equations are $ML^2=1,~M^3L^2=1$. Consider $(M,L)=(1,-1) \in \mathcal{S}_{\gamma,\delta}$ whose image in $\mathcal{B}_{\gamma,\delta}$ is the special point $s_{+-}=(2,-2,-2)$. From table \ref{fiber classification table for whitehead}, $C_{\gamma}\cap C_{\delta}\cap I_{+-}=\varnothing.$ This is an example where the intersection point in $X(T)$ occurs at one of the four singular points. The intersection point on $X(T)$ has nonempty fiber, but the fiber can never be an intersection point of $C_\gamma$ and $C_\delta$. 

The same phenomenon occurs at $s_{--}=(-2,-2,2)\in E_1\cap E_2$. For example, one may use the slopes $1/1$ and $3/1$.

\subsubsection{Case 6: a point of $E_1\cap E_3$}
Let $\gamma=\frac{2}{1},~\delta=\frac{2}{3}$. The peripheral equations are $M^2L=1,~M^2L^3=1$. Consider $(M,L)=(-1,1) \in \mathcal{S}_{\gamma,\delta}$ whose image in $\mathcal{B}_{\gamma,\delta}$ is the special point $s_{-+}=(-2,2,-2)$. From table \ref{fiber classification table for whitehead}, the fiber of this point, which is $R_{-+}=\left\{(-2,y,-y):y\in\mathbb C\right\}$, is contained in $C_{\gamma}\cap C_{\delta}$. This is an example in which the singular intersection point in $X(T)$ corresponds to a common curve component in $C_{\gamma}\cap C_{\delta}$. 

Note that the fiber of $s_{++}=(2,2,2)$ is contained in every $C_{\gamma}$, as remarked in \ref{R_++ contained in all C_gamma}. 

\subsubsection{Case 7: the point $e_0\in E_2\cap E_3$}
Let $\gamma=\frac{4}{1},~\delta=\frac{4}{3}$. The peripheral equations are $M^4L=1,~M^4L^3=1$. Consider $(M,L)=(i,1) \in \mathcal{S}_{\gamma,\delta}$ whose image in $\mathcal{B}_{\gamma,\delta}$ is the special point $e_{0}=(0,2,0)$. From table \ref{fiber classification table for whitehead}, the fiber of this point, which is $J_0=\left\{(0,y,0):y\in\mathbb C\right\}$, is contained in $C_{\gamma}\cap C_{\delta}$. This is another example in which the intersection point in $X(T)$ corresponds to a common curve component in $C_{\gamma}\cap C_{\delta}$.

\section{topological interpretation of intersection points}
In the previous section, we studied the way to systematically compute the intersection pattern $C_{\gamma} \cap C_{\delta} \subset X_0(M)$. In this section, we analyze the mathematical significance of the intersection points. As discussed in the introduction, the significance of studying the intersection pattern is at least twofold: to extract topological information of the corresponding 3-manifolds $M$ and to read off algebro-geometric properties of the canonical surface $X_0(M)$ itself. We study the topological implications about $M$ of the intersection points in this section, and we leave the research of the canonical surfaces using the intersection patterns for future work. Our main conclusion is that we can use the intersection points to detect common boundary slopes for both Dehn-filled manifolds. To begin with, we introduce the notions of boundary slope and ANI representation. Then, we focus on the Whitehead link and show how to use ANI intersection to compute the common boundary slopes. This section is inspired by \cite{Michelle}, where the author studies the topological implications of intersections between different components in the character varieties of a special family of 2-brdige knots. 

\subsection{boundary slopes and ANI representation}

An essential surface in a 3-manifold is a properly embedded orientable incompressible surface which is not boundary parallel. Let $M$ be a knot exterior in some homology 3-sphere and use $(\mu,\lambda)$ to denote a pair of preferred meridian and longitude on $\partial M$. Any essential surface $S$ with non-empty boundary in $\partial M$ has a collection of disjoint circles on the torus boundary, and each circle can be represented by $\mu^p\lambda^q$ for some relatively prime $p$ and $q$. $\frac{p}{q}$ is called the slope of $S$, and we say that $\frac{p}{q}$ is a boundary slope for $M$ if there is an essential surface $S$ in $M$ with slope $\frac{p}{q}$. We call $\frac{0}{1}$ the $0$-slope and $\frac{1}{0}$ the $\infty$-slope. 

Let $K$ be a field. Recall that a discrete valuation on $K$ is a homomorphism $v: K^{*} \longrightarrow \Z$ such that
\begin{equation}\label{property of valuation}
\begin{cases}
    v(a+b) \geq \operatorname{min}(v(a),v(b)),~~\text{for any} ~a,b \in K^{*}. \\
    v(a+b) = \operatorname{min}(v(a),v(b)),~~ \text{if} ~v(a)\neq v(b).
\end{cases}
\end{equation}
We extend $v$ to all of $K$ by assigning $v(0)=\infty$. Set $O_v =\{a \in K \mid v(a) \geq 0\}$, which forms a subring of $K$, called the valuation ring of $v$.

Let $\Gamma = \pi_1(M)$ be the fundamental group of $M$. From Culler-Shalen theory, there is a canonical way to construct essential surfaces associated to a representation $\rho : \Gamma \longrightarrow \operatorname{SL}(2,K)$, where $K$ is a field with a discrete valuation $v$ such that $v(\operatorname{tr}(\rho(\gamma)) < 0$ for some $\gamma \in \Gamma$. We summarize the main ideas here, and the reader may refer to \cite{Culler_Shalen},\cite{Shalen}, \cite{Schanuel_Zhang} \cite{Serre} for more references. The first step is to construct a simplicial tree $T_v$ on which $\operatorname{SL}(2,K)$ acts without inversion. The construction of $T_v$ is described in \cite{Culler_Shalen}$\S 2$, \cite{Shalen}$\S 3$, where the vertices of the tree are given by the homothety classes of lattices in $K^2$, and there is an edge between two vertices if there exists representative lattices $\Lambda_0$ and $\Lambda_1$ and a linear automorphism $A$ of $K^2$ such that $v(\operatorname{det}(A))=1$ and $A(\Lambda_0)=\Lambda_1$. It turns out that $\operatorname{SL}(2,K)$ acts on $T_v$ simplicially and without inversions by the action induced from the action on $K^2$.

When $\Gamma$ has a representation $\rho$ into the group $\operatorname{SL}(2,K)$, there is an induced action of $\Gamma$ on the tree $T_v$ via the representation $\rho$. If there is an element $\gamma \in \Gamma$ with $v(\operatorname{tr}(\rho(\gamma)) < 0$, the action of $\Gamma$ on $T_v$ will be nontrivial(i.e. no vertex of $T_v$ is fixed by $\Gamma$). By Culler-Shalen theory, if the action of $\Gamma$ on $T_v$ is nontrivial, there is an essential surface associated to this action(see \cite{Culler_Shalen}$\S 2$). We say such an essential surface is detected by the representation $\rho$. In particular, consider a representation $\rho : \Gamma \longrightarrow \operatorname{SL}(2,K)$ where $K$ is a number field and such that there is some element $\gamma \in \Gamma$ with $\operatorname{tr}(\rho(\gamma))$ not an algebraic integer. We call such a representation an algebraic non-integral(ANI) representation. Then there exists some discrete valuation $v$ so that $v(\operatorname{tr}(\rho(\gamma)))<0$ and an essential surface can be detected. Morever, we can use the following lemma(Corollary 3 of \cite{Schanuel_Zhang}) to determine the boundary slope detected by an ANI representation. 

\begin{lemma}\label{strongly detected boundary slope}
Let $M$ be an irreducible knot exterior such that $\pi_1(M)$ acts non-trivially on the tree defined by a discrete valuation $v$ of a field $K$ via a representation $\rho: \pi_1(M) \longrightarrow \operatorname{SL}(2,K)$. If $\operatorname{tr}(\rho(\delta)) \in O_v$ for one slope $\delta$ in $\partial M$ and $\operatorname{tr}(\rho(\gamma)) \notin O_v$ for another slope
$\gamma$ in $\partial M$,then $M$ contains an essential surface $S$ associated to the tree such that the boundary of $S$ is nonempty and is of the slope $\delta$,i.e. $\delta$ is a boundary slope.
\end{lemma}

The following lemma is very useful when explicitly computing the detected boundary slope.

\begin{lemma}\label{slope is -v(l)/v(m)}
Let $W$ be an irreducible knot exterior with a representation $\rho: \pi_1(W) \longrightarrow \operatorname{SL}(2,K)$, where $K$ is a field with a discrete valuation $v$. If $\operatorname{tr}(\rho(\delta)) \in O_v$ for one slope $\delta=\mu^{p}\lambda^{q}$ and $\operatorname{tr}(\rho(\gamma)) \notin O_v$ for another slope $\gamma$. Then $\frac{p}{q}=-\frac{v(L)}{v(M)}$, where $(M,L)$ is a compatible pair of eigenvalues of $\rho(\mu)$ and $\rho(\lambda)$. 
\end{lemma}

\begin{proof}
Since $\rho(\mu)=\operatorname{diag}(M,M^{-1}), ~\rho(\lambda)=\operatorname{diag}(L,L^{-1})$, we have $\rho(\delta)=\rho(\mu^p\lambda^q)=\operatorname{diag}(M^pL^q,M^{-p}L^{-q})$. Let $t=M^pL^q$, then $v(t)=pv(M)+qv(L)$ and $\operatorname{tr}(\rho(\delta))=t+t^{-1}$. If $v(t)\neq0$, then $v(t^{-1)}=-v(t)\neq v(t)$ and hence $v(\operatorname{tr}(\rho(\delta)))=v(t+t^{-1})=\operatorname{min}(v(t),v(t^{-1}))<0$. This contradicts the fact that $\operatorname{tr}(\rho(\delta)) \in O_v$. So we must have $v(t)=0$ and hence $pv(M)+qv(L)=0$. Since $v(\operatorname{tr}(\rho(\gamma)))<0$, $(v(M),v(L))\neq (0,0)$. Therefore, $\frac{p}{q}=-\frac{v(L)}{v(M)}$(we view $\infty = \frac{1}{0}$). 
\end{proof}

\begin{remark}\label{v(g)<0 is enough}
If $\rho$ is an ANI representation such that $\operatorname{tr}(\rho(\gamma)) \notin O_v$ for some slope $\gamma \in \pi_1(T)$ where $T$ is the boundary torus, then it can be shown that $\operatorname{tr}(\rho(\delta)) \in O_v$ for some other slope $\delta$. Therefore, the hypotheses in the previous two lemmas are automatically satisfied, once $\operatorname{tr}(\rho(\gamma)) \notin O_v$ for some slope $\gamma$. 
\end{remark}

For our purpose, we are interested in the case where the ANI representation comes from an intersection point $\chi \in C_\gamma \cap C_{\delta} \subset X_0(M)$, where $M$ is a 2-bridge link complement in $\mathbb{S}^3$. We denote the two cusps of $M$ by $T_1,T_2$, and suppose that the Dehn fillings with distinct slopes $\gamma=\frac{p}{q},~ \delta=\frac{r}{s}$ are performed on $T_1$. We stick to the previous conventions and choose the standard two generators $a,b$ for $\pi_1(M)$, where $a,b$ correspond to the meridians of $T_1,T_2$, respectively. Then the coordinates for $X_0(M)$ are given by $x=\operatorname{tr}\rho(a),~y=\operatorname{tr}\rho(b),~z=\operatorname{tr}\rho(ab)$. Let $\mu_i,\lambda_i$ be the meridian and preferred longitude for $T_i$(so $\mu_1=a,\mu_2=b$). Let $(m_i,\ell_i,n_i)$ be the coordinates of the character variety $X(T_i)$, where $m_i=\operatorname{tr}\rho(\mu_i),~
\ell_i=\operatorname{tr}\rho(\lambda_i),~n_i=\operatorname{tr}\rho(\mu_i\lambda_i) $. We first have the following observation. 

\begin{lemma}\label{traces on T_1 are algebraic integers}
Let $\chi \in C_\gamma \cap C_\delta$ and let $(m_1,\ell_1,n_1)=\phi(\chi)$, where $\phi: X_0(M) \longrightarrow X(T_1)$ is the boundary restriction map onto $T_1$. Then $m_1,\ell_1$ are both algebraic integers. 
\end{lemma}

\begin{proof}
Let $\rho$ be a representative of $\chi$. Write $P_{\mu_1}=\rho(\mu_1),~ P_{\lambda_1}=\rho(\lambda_1)$. These two matrices commute.  The two filling relations give $P_{\mu_1}^\Delta=(P_{\mu_1}^p P_{\lambda_{1}}^q)^s(P_{\mu_1}^r P_{\lambda_1}^s)^{-q}=I$ and $P_{\lambda_1}^\Delta=(P_{\mu_1}^r P_{\lambda_1}^s)^p(P_{\mu_1}^p P_{\lambda_1}^q)^{-r}=I.$ Thus the eigenvalues of $P_{\mu_1}$ and $P_{\lambda_1}$ are roots of unity. Their traces are therefore algebraic integers. 
\end{proof}

For any $g \in \pi_1(M)$, we have $\chi(g) \in \Z[x,y,z]$ by the trace identities(see lemma \ref{trace identities}). Therefore, $\chi$ is the character of an ANI representation if and only one of $x,y,z$ is not an algebraic integer. From the lemma above, for any $\chi \in C_\gamma \cap C_\delta$, $x=m_1$ must be an algebraic integer. So in order for the intersection character $\chi$ to be an ANI character(that is, the character of an ANI representation), one of the y-coordinate or z-coordinate of $\chi$ is not an algebraic integer. Typically, as shown in Proposition \ref{criterion for ANI} below for the Whitehead link, it is the y-coordinate of $\chi$ that is not an algebraic integer. Therefore, we are interested in those $\chi \in C_\gamma \cap C_{\delta}$ such that the y-coordinate of $\chi$ is not an algebraic integer. 

For such a character $\chi$, since $\chi \in C_\gamma \cap C_\delta$, it represents a character for the fundamental groups of both $M_\gamma$ and $M_\delta$(cf. subsection 2.3), whose restrictions to the remaining peripheral subgroup $\pi_1(T_2)$ are the same. Since the y-coordinate of $\chi$ is not an algebraic integer and the y-coordinate is the trace of $\rho(\mu_2)$, we can use $\chi$ and the arguments above to detect boundary slopes for both $M_\gamma$ and $M_\delta$. The slopes detected on the remaining cusp $T_2$ must be the same for both $M_\gamma$ and $M_\delta$, since they are detected using the same character $\chi$. Let $v$ be a discrete valuation such that $v(y)<0$. From lemma \ref{slope is -v(l)/v(m)} and remark \ref{v(g)<0 is enough}, the common boundary slope for $M_{\gamma}$ and $M_{\delta}$ associated to such an ANI $\chi \in C_\gamma \cap C_\delta$ can be computed by $-\frac{v(L_2)}{v(M_2)}$, where $(M_2,L_2)$ is a compatible pair of eigenvalues of $\rho(\mu_2)$ and $\rho(\lambda_2)$.

\subsection{boundary slopes detected from intersection points}
Based on the discussions in the previous subsection, we can use $\chi \in C_\gamma \cap C_\delta$ to detect common boundary slopes for $M_\gamma$ and $M_\delta$, provided that the y-coordinate of $\chi$ lies in some number field but is not an algebraic integer. In this subsection, we work on the Whitehead link specifically. We give a criterion to determine when an intersection character is ANI, and we explicitly compute the common boundary slopes that can possibly arise from the ANI intersection. 

Let $M$ be the Whitehead link complement in $\mathbb{S}^3$. As demonstrated in subsection 3.2, the canonical component is given by $X_0(M)=\{F(x,y,z)=0\}$, where $F(x,y,z)=z^3-xyz^2+(x^2+y^2-2)z-xy.$ Set $\kappa=\operatorname{tr}\rho([a,b])=x^2+y^2+z^2-xyz-2$, so the equation $F=0$ is equivalently $z\kappa=xy$. The formulas for $(m_1,\ell_1,n_1)$ in terms of $x,y,z$ are given in lemma \ref{trace functions of n and l}. The formulas for $(m_2,\ell_2,n_2)$ are obtained in a similar manner by exchanging $x$ and $y$ in lemma \ref{trace functions of n and l}. To summarize, we have 
\begin{equation}\label{trace formulas for m_i, l_i, n_i}
\begin{aligned}
    &m_1=x,\qquad \ell_1=\kappa^2-y^2\kappa+2y^2-2,\qquad n_1=(x-yz)\kappa-x+2yz\\
    &m_2=y,\qquad \ell_2=\kappa^2-x^2\kappa+2x^2-2,\qquad n_2=(y-xz)\kappa-y+2xz. 
\end{aligned}
\end{equation}

The proposition below provides a simple criterion of when a $\chi \in C_\gamma \cap C_\delta$ is ANI. Recall the definitions of $\mathcal{S}_{\gamma,\delta}$ and $\mathcal{B}_{\gamma,\delta}$ in \ref{def of B_gamma,delta}. 

\begin{proposition}\label{criterion for ANI}
Let $M$ be the Whitehead link complement in $\mathbb{S}^3$. Let $P=(m_1,\ell_1,n_1) \in \mathcal{B}_{\gamma,\delta} \subset X(T_1)$ such that $P$ lies inside $U_{fin}$, the finite locus of $\phi: X_0(M) \longrightarrow X(T_1)$. Define 
\begin{equation}\label{compute T,H}
    T=\frac{m_1(m_1+n_1)}{2-\ell_1+m_1n_1}, 
\qquad  
H=\begin{cases}
    \frac{\ell_1+2-T^2}{2-T}, ~~\text{if} ~~T\neq 2  \\
    4, ~~\text{if} ~~T=2.
    \end{cases}
\end{equation}
Suppose that $\chi=(x,y,z)$ lies in the fiber $\phi^{-1}(P)$. Then $\chi$ is the character of an ANI representation if and only if $y$ is not an algebraic integer, and $y$ is not an algebraic integer if and only if $H$ is not an algebraic integer. 
\end{proposition}

\begin{proof}
First, from the defining equation $F=0$ of $X_0(M)$(cf. \eqref{defining equation for X_0(M) Whitehead}), $z$ satisfies a polynomial equation with coefficients in $\Z[x,y]$ and with leading coefficient $1$. Therefore, by the transitivity of integrality, $z$ is an algebraic integer if $x,y$ are. Since $x$ is an algebraic integer by lemma \ref{traces on T_1 are algebraic integers}, $\chi$ is the character of an ANI representation if and only if $y$ is not an algebraic integer. 

Since $P$ lies in the finite locus of $\phi$, the fiber of $P$ is nonempty and we choose $\chi=(x,y,z) \in \phi^{-1}(P)$. From the defining equation of $X_0(M)$, the definition of $\kappa$ and \eqref{trace formulas for m_i, l_i, n_i}, we have 
\begin{equation}\label{equation for a fiber}
    \begin{cases}
        &\kappa=m_1^2+y^2+z^2-m_1yz-2, \\
        & z\kappa=m_1y,\\
        & \ell_1=\kappa^2+(2-\kappa)y^2-2\\
        & n_1=m_1(\kappa-1)+(2-\kappa)yz
    \end{cases}
\end{equation}

From \eqref{equation for a fiber}, we have $2-\ell_1+m_1n_1=m_1^2+(2-\kappa)z^2$ and  $m_1(m_1+n_1)=m_1^2\kappa+(2-\kappa)m_1yz$. Then 
$$
\kappa(2-\ell_1+m_1n_1)=\kappa m_1^2+(2-\kappa)\kappa z^2=m_1^2\kappa+(2-\kappa)m_1yz=m_1(m_1+n_1).
$$
Since the point $P$ lies in the finite locus of $\phi$, $2-\ell_1+m_1n_1 \neq 0$ and hence $\kappa=\frac{m_1(m_1+n_1)}{2-\ell_1+m_1n_1}=T$. When $T\neq 2$, from the third equation of \eqref{equation for a fiber}, we have $y^2=\frac{\ell_1+2-T^2}{2-T}=H$. When $T=\kappa=2$, from the first two equations in \eqref{equation for a fiber}, we have $z=\frac{m_1y}{2},~ 4=m_1^2+y^2-\frac{m_1^2y^2}{4}$. Therefore, $(m_1^2-4)(y^2-4)=0$. Since $P$ lies in the finite locus, $m_1^2-4\neq 0$. So we must have $y^2=4=H$. Collecting these results gives $y^2=H$ and hence $y$ is an algebraic integer if and only if $H$ is.  
\end{proof}

\begin{remark}\label{all are algebraic numbers}
We make a comment that when $P=(m_1,\ell_1,n_1) \in \mathcal{B}_{\gamma,\delta}$, $m_1,\ell_1,n_1$ are clearly all algebraic numbers, and so are $T,H,y,z$. So $\chi$ is indeed the character of some $\rho$ where the target of $\rho$ lies in $\operatorname{SL}(2,K)$ for some number field $K$. Therefore, it makes sense to discuss whether $\chi$ is ANI or not for $\chi \in \phi^{-1}(P)$. 
\end{remark}

\begin{example}\label{example 1 of section 4}
Consider $\gamma=\frac{p}{q}=\frac{1}{3},~\delta=\frac{r}{s}=\frac{1}{15}$. The peripheral equations are $M_{1}L^3_{1}=1,~M_{1}L_{1}^{15}=1$. Let $\zeta_{12}=e^{\pi i/6}$ and choose $(M_1,L_1)=(-i,\zeta_{12})$. Then $(M_1,L_1) \in \mathcal{S}_{\gamma,\delta}$ and the corresponding point in $\mathcal{B}_{\gamma,\delta}$ is given by $P=(m_1,\ell_1,n_1)=(0,\sqrt{3},1)\in U_{\mathrm{\acute et}}$. Substituting $P$ into \eqref{compute T,H} gives $T=0,H=\frac{\sqrt{3}+2}{2}$. From Proposition \ref{criterion for ANI} and discussions in the previous sections, any $\chi \in \phi^{-1}(P)$ will lie in $C_\gamma\cap C_\delta$ and will be the character of an ANI representation, with the negative valuation obtained from a 2-adic valuation. It can be computed that the detected boundary slope in this case is 0(cf. Theorem \ref{boundary slopes are 0 or 4}). 
\end{example}

\begin{example}\label{example 2 of section 4}
Consider $\gamma=\frac{p}{q}=\frac{1}{17},~\delta=\frac{r}{s}=\frac{1}{35}$. The peripheral equations are $M_{1}L^{17}_{1}=1,~M_{1}L_{1}^{35}=1$. Let $\zeta_{3}=e^{2\pi i/3}$ and choose $(M_1,L_1)=(\zeta_3,\zeta_{3})$. Then $(M_1,L_1) \in \mathcal{S}_{\gamma,\delta}$ and the corresponding point in $\mathcal{B}_{\gamma,\delta}$ is given by $P=(m_1,\ell_1,n_1)=(-1,-1,-1)\in U_{\mathrm{\acute et}}$. Then $T=\frac{1}{2},H=\frac{1}{2}$ and any $\chi \in \phi^{-1}(P)$ will be the character inside $C_\gamma  \cap C_\delta$ of an ANI representation, with the negative valuation obtained from a 2-adic valuation. 

Moreover, an ANI character in $C_\gamma \cap C_\delta$ with negative 3-adic valuation can also be obtained by considering $\zeta_{18}=e^{\pi i/9}$ and $(M_1,L_1)=(\zeta_{18},\zeta_{18})$. It can be computed that the detected boundary slopes for both valuations in this case are 4(cf. Theorem \ref{boundary slopes are 0 or 4}). 
\end{example}

The two slopes 0 and 4 detected in the previous two examples in fact exhaust all the possibilities, as shown in the next theorem. 

\begin{theorem}\label{boundary slopes are 0 or 4}
Let $M$ be the Whitehead link complement in $\mathbb{S}^3$. Let $\gamma=\frac{p}{q},~ \delta=\frac{r}{s},~ ps-qr\neq 0$ be two distinct slopes on $T_1$. Suppose that $\chi \in C_{\gamma}\cap C_{\delta} \subset X_0(M)$ is the character of an ANI representation $\rho:\pi_1(M)\longrightarrow \mathrm{SL}_2(K)$
for some number field $K$, i.e. $\rho(\mu_1)^p\rho(\lambda_1)^q=\rho(\mu_1)^r\rho(\lambda_1)^s=I$. Let $v$ be a discrete valuation at which $v(y)<0$(cf. Proposition \ref{criterion for ANI}). Then the common boundary slope detected by $\chi,v$ is either $0$ or $4$. More precisely, writing $\kappa=\operatorname{tr}\rho([a,b]),$ the detected slope is
$$
\begin{cases}
0,& v(\kappa) \geq 0,\\
4,& v(\kappa)<0.
\end{cases}
$$
\end{theorem}

\begin{proof}
From lemma \ref{slope is -v(l)/v(m)}, the boundary slope detected is of the form $-\frac{v(L_2)}{v(M_2)}$. Therefore, we need to find $v(L_2)$ and $v(M_2)$ respectively. Let $h:=v(y)<0$. Since $v(y)<0$, the matrix $\rho(\mu_2)=\rho(b)$ has distinct eigenvalues. Then, we may simultaneously diagonalize $\rho(\mu_2)$ and $\rho(\lambda_2)$, since they commute:
$$
\rho(\mu_2)=
\begin{pmatrix}
M_2&0\\
0&M_2^{-1}
\end{pmatrix},
\qquad
\rho(\lambda_2)=
\begin{pmatrix}
L_2&0\\
0&L_2^{-1}
\end{pmatrix}.
$$
Thus
\begin{equation}\label{eq:wh-eigenvalue-traces}
m_2=M_2+M_2^{-1},\qquad
\ell_2=L_2+L_{2}^{-1},\qquad
n_2=M_2L_2+(M_2L_2)^{-1}.
\end{equation}
For any nonzero $\xi$, if $v(\xi)\neq 0$, then
\begin{equation}
\label{eq:wh-trace-valuation}
v(\xi+\xi^{-1})=-|v(\xi)|.
\end{equation}
Indeed, the two summands have distinct valuations, so the
valuation of their sum is their minimum. Hence, we choose $M_2$ so that $v(M_2)=v(m_2)=v(y)=h<0.$

Suppose that $v(\kappa)\geq 0$, including $\kappa=0$. Since $x$ is integral, the formula $\ell_2=\kappa^2-x^2\kappa+2x^2-2$ shows that $v(\ell_2)\geq 0.$ It follows from \eqref{eq:wh-eigenvalue-traces} and \eqref{eq:wh-trace-valuation} that $v(L_2)=0.$ Indeed, a nonzero value of $v(L_2)$ would force $v(\ell_2)=-|v(L_2)|<0$. Therefore, the detected slope is 0.

Suppose now that $t:=v(\kappa)<0.$ Since $v(2)\geq 0$, we have $v(2-\kappa)=t.$
Consider the equation $\ell_1+2=\kappa^2+(2-\kappa)y^2$ obtained from \eqref{trace formulas for m_i, l_i, n_i}. Its left-hand side is integral since $v(\ell_1)\geq 0$. On the right, $v(\kappa^2)=2t<0,~v((2-\kappa)y^2)=t+2h<0.$ These valuations must be equal. Otherwise the right-hand side would have a unique lowest-valuation term, and hence negative valuation, contradicting the integrality of the left-hand side.
Thus $2t=t+2h,$ and therefore $h=\frac{t}{2}.$

Next, from \eqref{trace formulas for m_i, l_i, n_i}, $\ell_2=\kappa^2-x^2\kappa+2x^2-2$. $v(\kappa^2)=2t$ while $v(x^2\kappa)\geq t>2t$. The remaining terms are integral. Hence $\kappa^2$ is the unique lowest-valuation term, giving $v(\ell_2)=2t.$ To determine the compatible sign of $v(L_2)$, we also need the mixed trace $n_2$. Since $\kappa\neq 0$, we have $z=\frac{xy}{\kappa}.$ Substituting this into \eqref{trace formulas for m_i, l_i, n_i}, we obtain
$$
n_2=y\kappa-xz\kappa-y+2xz=y\left(\kappa-x^2-1+\frac{2x^2}{\kappa}\right).
$$
Inside the parentheses, the term $\kappa$ has valuation $t<0$.
All other terms are integral; in particular, $v\left(\frac{2x^2}{\kappa}\right)=v(2)+2v(x)-t\geq -t>0.$ Therefore $v\left(\kappa-x^2-1+\frac{2x^2}{\kappa}\right)=t.$ Using $h=\frac{t}{2}$, we conclude that
\begin{equation}
\label{eq:wh-U-value}
v(n_2)=h+t=\frac{3t}{2}.
\end{equation}

Since $v(\ell_2)=2t$, the longitude eigenvalue satisfies $v(L_2)\in\{2t,-2t\}.$ We have chosen $v(M_2)=h=\frac{t}{2}.$ If $v(L_2)=2t$, then $v(M_2L_2)=\frac{5t}{2}<0.$ Equation \eqref{eq:wh-eigenvalue-traces} would imply $v(n_2)=\min\left\{\frac{5t}{2},-\frac{5t}{2}\right\}=\frac{5t}{2},$ contradicting \eqref{eq:wh-U-value}. Hence the compatible eigenvalue must satisfy $v(M_2)=\frac{t}{2},~v(L_2)=-2t.$ Then the detected slope is given by $-\frac{v(L_2)}{v(M_2)}=4$.
\end{proof}

\begin{remark}
Although there are only two possibilities for the detected boundary slopes for Dehn fillings of the Whitehead link, there is no restriction on the possible primes that lead to non-integral valuations. For instance, in example \ref{example 2 of section 4}, both the 2-adic and 3-adic valuations give non-integral valuations at some characters, but the detected boundary slopes for both of them are 4. 
\end{remark}

\begin{remark}
The $\frac{1}{n}$-fillings on $T_1$ of the Whitehead link will give twist knots in $\mathbb{S}^3$. From \cite{HT1985}, the possible boundary slopes for twist knots can be computed using topological methods. Except for the unknot and the trefoil knot, every twist knot has three distinct boundary slopes: 0,4, and another one. This last boundary slope is distinct for different twist knots. Therefore, the common boundary slopes for twist knots will typically be 0 and 4, which is consistent with our findings(our result is more general in the sense that the filling slopes can be arbitrary). It is surprising that similar conclusions can be reached by two completely different methods(one topological and one arithmetic). This suggests that there is a deeper mathematical connection here that deserves further investigation.
\end{remark}

We end this section with the following observation, which states that for Whitehead link, only irreducible intersection characters can detect the (strict) common boundary slope 4. 

\begin{proposition}\label{reducible only detect 0}
Let $M$ be the Whitehead link complement in $\mathbb{S}^3$. Suppose that $\chi \in C_{\gamma}\cap C_{\delta}$ is the character of a reducible and ANI representation. Then the detected slope is exactly $0$. 
\end{proposition}

\begin{proof}
Let $\chi$ be the character of $\rho: \pi_1(M) \longrightarrow \operatorname{SL}(2,K)$ for some number field $K$. Since $\rho$ is reducible, we can assume $\rho$ is upper-triangular, and hence diagonal(setting the upper-right corner to be 0 won't change the character). Therefore, we write 
$$
\rho(g)=
\begin{pmatrix}
\alpha(g) & 0\\
0 & \alpha(g)^{-1}
\end{pmatrix},
$$
for some group homomorphism $\alpha: \pi_1(M) \longrightarrow K^{\times}$. Since $K^{\times}$ is abelian, $\alpha$ factors through the abelianization of $\pi_1(M)$, which is $H_1(M;\Z) \cong \Z^2$. This means that there is a group homorphism $\overline{\alpha}: H_1(M;\Z) \longrightarrow K^{\times}$ such that $\alpha(g)=\overline{\alpha}([g])$ for any $g\in \pi_1(M)$, where $[g] \in H_1(M;\Z)$ represents the homology class of $g$. Since the linking number for the Whitehead link(between the two cusps) is 0, $[\lambda_1]=[\lambda_2]=0 \in H_1(M;\Z)$. Therefore, $\alpha(\lambda_2)=1$ and hence $L_2=1$. So $v(L_2)=0$ for any discrete valuation $v$. On the other hand, since $\chi \in C_{\gamma}\cap C_{\delta}$ is ANI, the negative valuation must occur at the character of $\mu_2$ from our previous discussions. Therefore, $v(\operatorname{tr}(\rho(\mu_2))=v(M_2+M_2^{-1})<0$ and hence $v(M_2)\neq 0$. Then the detected slope is $-\frac{v(L_2)}{v(M_2)}=0$. 
\end{proof}

\section{Classification of reducible intersections}
Reducible and irreducible intersection characters in $C_\gamma\cap C_\delta \subset X_0(M)$ can exhibit very different properties, as shown in Proposition \ref{reducible only detect 0}. Therefore, it is useful to find a way to distinguish reducible intersection characters from irreducible ones. In this section, we classify reducible intersection characters using twisted Alexander polynomials. Note that $C_\gamma$ may share a common line with $C_\delta$, which consists of entirely reducible characters, as the line $R_{++}$ discussed in remark \ref{R_++ contained in all C_gamma}. These are not the reducible intersections we are interested in. We care about reducible intersection $\chi$ that arises from the intersection of $C_1\subset C_\gamma$ and $C_2\subset C_\delta$, where $C_1$ and $C_2$ are curves containing the character of an irreducible representations, e.g. the canonical curves of $X(M_\gamma)$ and $X(M_\delta)$. For these reducible characters, they are characters of reducible representations with non-abelian images(cf. \cite{Chin_Reid_Stover} Proposition 2.11). There is a close connection between such representations and (twisted) Alexander polynomials, cf. \cite{HeusenerPorti2005}, \cite{Miller}. 

Our main result in this section is to establish a link between the reducible intersections in $C_\gamma \cap C_{\delta}$ and the common roots of the corresponding twisted Alexander polynomials, which is stated in Proposition \ref{reducible intersection and common root}. As a corollary, we show that the canonical curves of twist knots, when embedded in $X_0(M)$ of the Whitehead link, can never intersect at certain reducible character(Corollary \ref{never intersect at reducible for whitehead}). We comment that the reason for using twisted Alexander polynomial is that when $\gamma=\frac{p}{q}$ for $p\neq 1$, the Dehn-filling manifold $M_\gamma$ is a knot complement in some lens space, and the first homology group may have torsions. Therefore, we need the twisted Alexander polynomial to capture the torsion information. In the following subsection, we begin with an introduction to group cohomology and twisted Alexander polynomials, which are main ingredients of our proofs.

\subsection{Preliminaries}

\subsubsection{group cohomology}\label{group cohomology}

We first give a short introduction to group cohomology and we refer the reader to \cite{Gille} $\S 3$ for more references. Let $G$ be a group and let $A$ be an abelian group equipped with a left action of $G$,
$$
G\times A\longrightarrow A,
\qquad
(g,a)\longmapsto g\cdot a,
$$
such that every $g\in G$ acts by an automorphism of $A$.
Equivalently, $A$ is a left module over the integral group ring $\Z G.$ The group cohomology of $G$ with coefficients in $A$ is denoted by $H^n(G;A).$ There are two standard definitions:
\begin{enumerate}
    \item the cochain definition using maps $G^n\to A$;
    \item the resolution definition using a free or projective
    resolution of the trivial $\Z G$-module $\Z$.
\end{enumerate}
These two definitions are naturally equivalent.

For the cochain definition, let $n\geq 0$, define $C^n(G;A)=\Map(G^n,A)$, which is the set of all maps from $G^n$ to $A$. In particular, since $G^0$ is a one-point set, $C^0(G;A)=A.$

The coboundary operator $\delta^n:C^n(G;A)\longrightarrow C^{n+1}(G;A)$ is defined by
$$
\begin{aligned}
(\delta^n f)(g_1,\ldots,g_{n+1})
={}&g_1\cdot f(g_2,\ldots,g_{n+1})\\
&+\sum_{i=1}^{n}(-1)^i
f(g_1,\ldots,g_i g_{i+1},\ldots,g_{n+1})\\
&+(-1)^{n+1}f(g_1,\ldots,g_n).
\end{aligned}
$$
It can be easily verified that $\delta^{n+1}\circ\delta^n=0$. Thus
$$
0\longrightarrow C^0(G;A)
\xrightarrow{\delta^0}
C^1(G;A)
\xrightarrow{\delta^1}
C^2(G;A)
\xrightarrow{\delta^2}\cdots
$$
is a cochain complex. Define $Z^n(G;A)=\ker\delta^n,~B^n(G;A)=\operatorname{im}\delta^{n-1}$. The $n$-th group cohomology is
$$
H^n(G;A)=Z^n(G;A)/B^n(G;A).
$$

When the degree is 0, $(\delta^0a)(g)=g\cdot a-a$ for $a\in A$. Therefore,
$$
H^0(G;A)
=
A^G
=
\{a\in A:g\cdot a=a\text{ for every }g\in G\}.
$$

When the degree is 1, a $1$-cochain is a map
$u:G\longrightarrow A$. Its coboundary is $(\delta^1u)(g,h)=g\cdot u(h)-u(gh)+u(g)$. Thus $u$ is a $1$-cocycle if and only if $u(gh)=u(g)+g\cdot u(h).$ A $1$-coboundary has the form $u(g)=g\cdot a-a$ for some $a\in A$. Hence
$$
H^1(G;A)
=
\frac{
\{u:G\to A:u(gh)=u(g)+g\cdot u(h)\}
}{
\{u(g)=g\cdot a-a:a\in A\}
}.
$$
If the action is trivial, then $H^1(G;A)\cong\Hom(G,A)$.

We can also define group cohomology using free(projective) resolutions. Regard $\Z$ as a trivial left $\Z G$-module. Choose a free(projective) resolution
$$
\cdots
\xrightarrow{\partial_3}P_2
\xrightarrow{\partial_2}P_1
\xrightarrow{\partial_1}P_0
\xrightarrow{\varepsilon}\Z
\longrightarrow0,
$$
where every $P_n$ is a free(projective) left $\Z G$-module and the chain is exact at each $P_n$ and $\Z$. Applying the functor $\Hom_{\Z G}(-,A)$
produces a cochain complex
$$
0\longrightarrow
\Hom_{\Z G}(P_0,A)
\xrightarrow{\delta^0}
\Hom_{\Z G}(P_1,A)
\xrightarrow{\delta^1}
\Hom_{\Z G}(P_2,A)
\longrightarrow\cdots,
$$
where $\delta^n(\varphi)=\varphi\circ\partial_{n+1}.$ Then group cohomology is defined by
$$
H^n(G;A)
=
H^n\!\left(\Hom_{\Z G}(P_*,A)\right).
$$

Using the bar resolution, it can be shown that the two definitions of group cohomology are in fact equivalent. 

\begin{theorem}[Equivalence of the two definitions]
Let $G$ be a group and let $A$ be a left $\Z G$-module. Let $P_*\longrightarrow\Z$
be any projective, and in particular any free, resolution of the trivial $\Z G$-module $\Z$. Then the cochain definition and the free-resolution definition(using $P_{*}$) determine the same group cohomology groups.
\end{theorem}

\subsubsection{Alexander polynomial and twisted Alexander polynomial}\label{twisted Alexander module and polynomials}
We now introduce the notions of (ordinary) Alexander polynomial and twisted Alexander polynomial. \cite{Zieschang} $\S 8$ provides a comprehensive discussion of Alexander polynomial. For more references on the twisted one, the reader may refer to \cite{HeusenerPorti2005} $\S 2$. Let $\Lambda=\mathbb C[t,t^{-1}].$ This is a principal ideal domain, and its units are $\Lambda^\times=\{ct^k : c\in\mathbb C^\times,\ k\in\mathbb Z\}.$ Let $M$ be a finitely generated $\Lambda$-module. By the structure theorem for finitely generated modules over a PID, there is a decomposition $M\cong\Lambda^r\oplus\bigoplus_{j=1}^s \Lambda/(f_j(t))$, and one may choose the invariant factors so that $f_1(t)\mid f_2(t)\mid\cdots\mid f_s(t).$ The module $M$ is called a torsion $\Lambda$-module if and only if $r=0$. Throughout the paper, we always assume the $\Lambda$-module in consideration is torsion. 

If $M$ is torsion, its order is defined by
$$
\operatorname{ord}_{\Lambda}(M)
\doteq
\prod_{j=1}^s f_j(t),
$$
where $\doteq$ denotes equality up to multiplication by a unit
$ct^k\in\Lambda^\times$. In particular, $\operatorname{ord}_{\Lambda}(0)=1.$

Equivalently, if $M$ admits a finite presentation
$$
\Lambda^m
\xrightarrow{A(t)}
\Lambda^n
\longrightarrow
M
\longrightarrow 0,
$$
then the zeroth fitting ideal $\operatorname{Fitt}_0(M)$ is generated
by the $n\times n$ minors of $A(t)$. Since $\Lambda$ is a PID, this
ideal is principal, and its generator, defined up to a unit, is
$\operatorname{ord}_{\Lambda}(M)$.

Now, let $X$ be a compact orientable three-manifold with torus boundary
such that $H_1(X;\Z)\cong \Z\oplus T,$ where $T$ is finite. We fix $\phi:\pi_1(X)\longrightarrow H_1(X;\Z)/T\cong\Z$ to be the homomorphism onto the free part. The associated infinite cyclic cover is denoted $X_\phi\longrightarrow X.$ The module $H_1(X_\phi;\C)$ is naturally a module over $\Lambda=\C[t^{\pm1}]$, where the action of $t$ is induced by the deck transformation. 

\begin{definition}\label{def of ordinary Alexander}
The $\Lambda$-module $H_1(X_\phi;\C)$ is called the ordinary Alexander module of $X$, and the ordinary Alexander polynomial of $X$, denoted $\Delta_{X}^{1}(t)$, is defined to be the order of $H_1(X_\phi;\C)$.  
\end{definition}

The ordinary Alexander polynomial forgets the finite group $T$. To retain the torsion information, fix a splitting $H_1(X;\Z)\cong\Z\oplus T$, and write the homology class of \(g\in\pi_1(X)\) as $[g]=\bigl(\phi(g),\tau(g)\bigr) \in\Z\oplus T$. Let $\sigma:T\longrightarrow\C^{\times}$ be a character. Let $G=\pi_1(X)$ and define $\Phi_\sigma:\C G\longrightarrow\C[t^{\pm1}]$ by
$$
\Phi_\sigma(g)=\sigma(\tau(g))t^{\phi(g)}.
$$
where $\C G$ is the group ring of $G$ with coefficients in $\C$ and $\Phi_\sigma$ extends linearly over $\C$. We define $\Lambda_{\sigma}$ to be $\Lambda_{\sigma}=\Lambda=\C [t,t^{-1}]$, regarded as a right $\mathbb CG$-module by $a\cdot g=a\Phi_{\sigma}(g)$. Note that $\Lambda_{\sigma}$ is also naturally equipped with a left $\Lambda$-action by self-multiplication. 

To see how the ordinary Alexander module(polynomial) can be generalized to incorporate torsion information, we first make the following observation when $\sigma$ is trivial. For a CW complex $X$, We use $C_*(X;-)$ to denote its standard chain complex and $\widetilde{X}$ its universal cover. 

\begin{lemma}\label{relation between chain complex of X_phi and tilde(X)}
Let $\sigma=1$ and regard $\Lambda$ as a right $\C G$-module by $a\cdot g=a\,t^{\phi(g)},~ a\in\Lambda,~ g\in G$. Then there is a natural isomorphism of chain complexes of left $\Lambda$-modules 
$$
\boxed{C_*(X_\phi;\C) \cong \Lambda\otimes_{\C G}C_*(\widetilde X;\C).}
$$
\end{lemma}

\begin{proof}
Let $q:\widetilde X\longrightarrow X_\phi$ be the quotient map. If $c$ is a cellular chain in $\widetilde X$, write $\overline c=q_\#(c).$ With our convention for the generator $t$, we have $\overline{gc}=t^{\phi(g)}\overline c.$ Define $\Psi_n: \Lambda\otimes_{\C G}C_n(\widetilde X;\C) \longrightarrow C_n(X_\phi;\C)$ by $\Psi_n(a\otimes c)=a\overline c.$ The balancing relation in the tensor product is $(a\cdot g)\otimes c=a\otimes gc.$ Now,
$$
\Psi_n((a\cdot g)\otimes c) = \Psi_n(at^{\phi(g)}\otimes c) = at^{\phi(g)}\overline c,
$$
whereas $\Psi_n(a\otimes gc)=a\overline{gc}=at^{\phi(g)}\overline c.$ Thus $\Psi_n$ is well-defined. It also commutes with the boundary operators:
$$
\partial\Psi_n(a\otimes c) = \partial(a\overline c) = a\,\partial\overline c
= a\,\overline{\partial c} = \Psi_{n-1}(a\otimes\partial c).
$$
Moreover, the lifts of every cell of $X$ to $X_\phi$ form a free
orbit under $\Z=\langle t\rangle$. Therefore, $\Psi_n$ is an
isomorphism for every $n$, proving that $C_*(X_\phi;\C) \cong \Lambda\otimes_{\C G}C_*(\widetilde X;\C).$ 
\end{proof}

Now let $\sigma$ be a (possibly)nontrivial character of $T$. We define $C_*^{\sigma}(X_\phi;\mathbb C):=\Lambda_{\sigma}\otimes_{\mathbb CG} C_*(\widetilde X;\mathbb C)$, which is a chain of left $\Lambda$-modules whose $\Lambda$-action is induced by self-multiplication on $\Lambda_\sigma$. Then we have the following definition. 

\begin{definition}\label{def of twisted Alexander}
the $\sigma$-twisted Alexander module of $X$ is defined by $A_X^{\sigma}:=H_1\bigl(C_*^{\sigma}(X_\phi;\mathbb C)\bigr)$, and the $\sigma$-twisted Alexander polynomial of $X$ is defined by $\Delta_{X}^{\sigma}(t) \doteq\operatorname{ord}_{\Lambda}(A_X^\sigma).$
\end{definition}

From Lemma \ref{relation between chain complex of X_phi and tilde(X)}, it is clear that $A_{X}^{1}$ is the ordinary Alexander module and hence the definitions are natural generalizations of those given in Definition \ref{def of ordinary Alexander}.

\begin{remark}
There are other conventions when define twisted Alexander polynomials such as the use of Wada's normalization or Reidemeister torsion. Such invariants may differ from our definition by certain factors, but they typically won't affect the zero set of the twisted Alexander polynomial. Since we are mainly interested in the zero set of the Alexander polynomial, we shall stick to our definitions without any ambiguity. 
\end{remark}

Finally, we introduce the Alexander polynomial for link complement in $\mathbb{S}^3$ with two boundary components. In this case, we still consider the infinite cyclic cover of $X_\phi$ of $X$ and define the Alexander module as $A_X=H_1(X_{\phi};\C)$. But now $\Lambda=\C[t_1^{\pm 1},t_2^{\pm 1}]$ is not a PID and hence the structure theorem for PID can not be applied. Therefore we consider the finite presentation $\Lambda^m\xrightarrow{A(t)} \Lambda^n \longrightarrow A_X \longrightarrow 0$. The zeroth fitting ideal $\operatorname{Fitt}_0(A_X)$ is now of the form $\operatorname{Fitt}_0(A_X)=(f_1,f_2,...,f_k)$. We define the Alexander polynomial of $X$ to be the gcd of the generators. In summary, we have 
\begin{definition}\label{def of Alexander for 2-components}
Let $X$ be a link complement in $\mathbb{S}^3$ with two boundary components. The The Alexander module of $X$ is the $\Lambda$-module $H_1(X_\phi;\C)$, and the Alexander polynomial of $X$, denoted $\Delta_{X}(t_1,t_2)$, is defined to be $
\Delta_{X}(t_1,t_2)\doteq \operatorname{gcd}(f_1,f_2,...,f_k)$. 
\end{definition}

\begin{remark}\label{Alexander poly for 2-bridge link}
If $X$ is a 2-bridge link complement in $\mathbb{S}^3$. It can be shown that $\operatorname{Fitt}_0(A_X)=(\Delta_X)$ is principal and $A_X\cong\Lambda/(\Delta_X)$. 
\end{remark}

\subsection{Classification of reducible intersections using twisted Alexander polynomials}
In this subsection, we derive the Burde--de Rham criterion(Theorem \ref{twisted Burde de Rham theorem}) for knot exteriors in lens spaces, which establishes the connection between reducible representations with nonabelian images and the roots of twisted Alexander polynomials. The key idea of the proof is that we use certain cohomology group as an intermediary to connect the two, cf. Theorem \ref{Cohomological criterion} and \ref{equivalence of H^1 and twisted Alexander}. We then generalize the result to 2-cusp link exterior in $\mathbb{S}^3$, i.e. Theorem \ref{Burde de Rham theorem for 2-bridge link}. Finally, we apply the criterion to reducible intersections in $C_\gamma\cap C_\delta$ to derive Proposition \ref{reducible intersection and common root} and Corollary \ref{never intersect at reducible for whitehead}. 

We comment that similar results of the twisted version of Burde-de Rham criterion have been obtained in \cite{HeusenerPorti2005}, \cite{Miller}, where the proofs rely on extensive computations involving Fox calculus. We present an alternative proof that makes greater use of techniques from algebraic topology(Tensor-Hom adjunction, Universal Coefficient Theorem etc.). We hope that this proof can provide a clearer perspective for understanding this theorem from the viewpoint of topology. Moreover, our proof can be easily generalized to obtain the 2-cusp version of the Burde-de Rham criterion. 

We begin with some general facts about group cohomology. Let $G$ be an arbitrary group, and let $\chi:G\longrightarrow \C^\times$
be a character, i.e. a group homomorphism. Denote by $\C_\chi$ the one-dimensional
$G$-module $\C$ whose action is given by $g\cdot z=\chi(g)z$. From the definitions given in section \ref{group cohomology}, a $1$-cocycle with coefficients in $\C_\chi$ is a map $u:G\longrightarrow \C$
satisfying $u(gh)=u(g)+\chi(g)u(h)$ and a $1$-coboundary is a cocycle of
the form $u(g)=c\bigl(1-\chi(g)\bigr)$ for some $c\in\C$. Thus, $H^1(G;\C_\chi)\neq 0$ means that there exists a $1$-cocycle which is not a
coboundary.

Now suppose that $\rho:G\longrightarrow \operatorname{SL}(2,\C)$ is a reducible representation. Then $\rho(G)$ preserves a one-dimensional subspace of
$\C^2$. After choosing a basis adapted to that invariant line, we may
write
$$
\rho(g)=
\begin{pmatrix}
\alpha(g) & b(g)\\
0 & \alpha(g)^{-1}
\end{pmatrix},
$$
where $\alpha:G\longrightarrow\C^\times,~b: G\longrightarrow\C$ are maps. Let $u(g)=\alpha(g)b(g)$. To ensure that $\rho$ is a group homomorphsim, we need the following lemma. 

\begin{lemma}\label{condition for rho to be group homo}
The map \(\rho\) is a group homomorphism if and only if $\alpha$ is a character and $u(gh)=u(g)+\alpha(g)^2u(h)$.
\end{lemma}

\begin{proof}
With the substitution $u(g)=\alpha(g)b(g)$, we clearly have 
$$
\rho(gh)=
\begin{pmatrix}
\alpha(gh) & u(gh)\alpha(gh)^{-1}\\
0 & \alpha(gh)^{-1}
\end{pmatrix}.
$$

Then by direct computation, we get 
$$
\rho(g)\rho(h)
=
\begin{pmatrix}
\alpha(g)\alpha(h) &
\alpha(g)u(h)\alpha(h)^{-1}
+
u(g)\alpha(g)^{-1}\alpha(h)^{-1}\\
0 & \alpha(g)^{-1}\alpha(h)^{-1}
\end{pmatrix}.
$$
Equalize the two matrices entry by entry, it can be verified that $\rho$ is a group homomorphism if and only if $\alpha$ is a character and $u(gh)=\alpha(g)^2 u(h) + u(g)$.
\end{proof}

From Lemma \ref{condition for rho to be group homo} and the condition that $\rho$ is a group homomorphism, we have $u(gh)=u(g)+\alpha(g)^2u(h)$. This means precisely that $u$ is a \(1\)-cocycle with coefficients in $\C_{\alpha^2}$ where the action on $\C_{\alpha^2}$ is $g\cdot z=\alpha(g)^2z$, i.e. $u\in Z^1(G;\C_{\alpha^2})$. Then we have the following, which is Lemma 5.1 in \cite{Miller}. We present the proof here for completeness.

\begin{theorem}\label{Cohomological criterion}
Let $\alpha:G\longrightarrow\C^\times$ be a character. Then there exists a reducible representation $\rho:G\longrightarrow SL(2,\C)$ with nonabelian image and diagonal characters $\alpha$ and $\alpha^{-1}$ if and only if $\alpha^2\neq 1$
and $H^1(G;\C_{\alpha^2})\neq 0.$
\end{theorem}

\begin{proof}
Let $u\in Z^1(G;\C_{\alpha^2}),$ so that $u(gh)=u(g)+\alpha(g)^2u(h).$ Define
$$
\rho_u(g)=
\begin{pmatrix}
\alpha(g) & u(g)\alpha(g)^{-1}\\
0 & \alpha(g)^{-1}
\end{pmatrix}.
$$
By Lemma \ref{condition for rho to be group homo}, $\rho_u$ is a reducible $SL(2,\C)$-representation. After a suitable conjugation, we may assume that any reducible representation $\rho$ is of the form $\rho=\rho_u$ for some $u$. Suppose next that $u$ is a coboundary, so that $u(g)=c\bigl(1-\alpha(g)^2\bigr)$
for some $c\in\C$. Let
$$
P_c=
\begin{pmatrix}
1 & c\\
0 & 1
\end{pmatrix}.
$$
A direct calculation gives
$$
P_c
\begin{pmatrix}
\alpha(g) & 0\\
0 & \alpha(g)^{-1}
\end{pmatrix}
P_c^{-1}
=
\begin{pmatrix}
\alpha(g) &
c\bigl(1-\alpha(g)^2\bigr)\alpha(g)^{-1}\\
0 & \alpha(g)^{-1}
\end{pmatrix}.
$$
The right-hand side is $\rho_u(g)$. Hence $\rho_u$ is conjugate
to the diagonal representation
$$
\rho_\alpha(g)=
\begin{pmatrix}
\alpha(g) & 0\\
0 & \alpha(g)^{-1}
\end{pmatrix}.
$$
In particular, $[u]=0$ implies that $\rho_u(G)$ is abelian.

Conversely, suppose that $\alpha^2\neq 1$ and that $\rho_u(G)$ is abelian. For all $g,h\in G$, we have $\rho_u(g)\rho_u(h)=\rho_u(h)\rho_u(g).$ Directly compute both sides and Compare the upper-right entries, we have 
$$
u(g)+\alpha(g)^2u(h)
=
u(h)+\alpha(h)^2u(g).
$$
Hence $\bigl(1-\alpha(h)^2\bigr)u(g)=\bigl(1-\alpha(g)^2\bigr)u(h).$ Since $\alpha^2\neq 1$, there is $h_0\in G$ such that
$\alpha(h_0)^2\neq 1.$ Set $c=\frac{u(h_0)}{1-\alpha(h_0)^2}.$
Taking $h=h_0$ in the preceding identity gives $u(g)=c\bigl(1-\alpha(g)^2\bigr)$
for all $g\in G$. Thus $u$ is a coboundary, and therefore $[u]=0\in H^1(G;\C_{\alpha^2}).$
It follows that, when $\alpha^2\neq 1$, $\rho_u(G)$ is nonabelian if and only if $[u]\neq 0.$

Finally, suppose that $\alpha^2=1$. Then the cocycle condition becomes $u(gh)=u(g)+u(h),$ and
$$
\rho_u(g)
=
\alpha(g)
\begin{pmatrix}
1 & u(g)\\
0 & 1
\end{pmatrix}.
$$
Scalar matrices commute with all matrices, and the upper-triangular
unipotent matrices commute with one another. Hence $\rho_u(G)$ is
always abelian when $\alpha^2=1$. This proves the theorem.
\end{proof}

From Theorem \ref{Cohomological criterion}, to study the reducible representations of $G$ into $\operatorname{SL}(2,\C)$ with nonabelian image, it suffices to study characters $\alpha: G \to \C^{\times}$ with $\alpha^2 \neq 1$ and $H^1(G;\C_{\alpha^2}) \neq 0$. We now restrict our attention to the case where $G=\pi_1(X)$ for some 3-manifold $X$. When $G=\pi_1(X)$ for some CW complex $X$, the group cohomology of $G$ is closely related to the simplicial cohomology of the universal cover $\widetilde{X}$ of $X$, where algebraic topology techniques can be used for theoretical and computational purposes. The main result connecting the two types of cohomology theory is the following. 

\begin{lemma}\label{Computing group cohomology using universal cover}
Let $X$ be a path-connected CW complex with $G=\pi_1(X)$, and let $M$ be a left $\C G$-module. Then 
\begin{equation}\label{group cohomology using universal cover}
    H^1(G;M) \cong H^1\left(\Hom_{\C G}(C_*(\widetilde X;\C),M)\right).
\end{equation}
\end{lemma}

\begin{proof}
Let $P_*=C_*(\widetilde X;\C)=C_*(\widetilde X;\Z) \otimes_{\Z} \C .$ Let $e^{n}_i$ be a $n$-cell of $X$ and let $\tilde{e}^{n}_{i}$ be a lift of $e^{n}_i$ in $\widetilde{X}$. Since the action of $G$ on the lifted cells is free, every $P_n$ is a free $\C G$-module with basis $\{ \tilde{e}^{n}_{i}\}$. Hence there is an augmented chain complex
\begin{equation}\label{augmented chain complex}
    \cdots \longrightarrow P_2 \xrightarrow{\partial_2} P_1 \xrightarrow{\partial_1} P_0 \xrightarrow{\varepsilon} \C \longrightarrow0,
\end{equation}
where $\C$ is regarded as the trivial $G$-module, and $\epsilon$ is the augmentation map given by $\sum_{\tilde{v}} a_{\tilde{v}}\tilde{v} \mapsto \sum_{\tilde{v}} a_{\tilde{v}}$. 

For degree-one cohomology, since $\widetilde X$ is connected and simply connected, $H_0(\widetilde X;\C)\cong\C,~ H_1(\widetilde X;\C)=0.$ Hence
$$
P_2\longrightarrow P_1\longrightarrow P_0
\longrightarrow\C\longrightarrow0
$$
is exact at $P_1$, $P_0$, and $\C$. It is therefore a free
partial resolution sufficient for computing first group
cohomology, where the free resolution definition of group cohomology as presented in section \ref{group cohomology} is used. Consequently, $H^1(G;M) \cong H^1\left(\Hom_{\C G}(C_*(\widetilde X;\C),M)\right).$
\end{proof}

\begin{remark}
    If $X$ is a $K(G,1)$, then $\widetilde X$ is contractible. Consequently, the augmented chain complex \eqref{augmented chain complex} is a free $\C G$-resolution of the trivial $\C G$-module $\C$. Therefore, for every left $\C G$-module $M$, $H^n(G;M) \cong H^n \left(\Hom_{\C G}(P_*,M)\right)$ for all $n \geq 1$. But since we only need computation for $H^1$, it is not necessary to require $X$ to be a $K(G,1)$. 
\end{remark}

Now, let $X$ be a compact orientable 3-manifold with torus boundary
such that $H_1(X;\Z)\cong \Z\oplus T,$ where $T$ is finite. Fix a splitting $H_1(X;\Z)\cong\Z\oplus T$, and write the homology class of \(g\in\pi_1(X)\) as $[g]=\bigl(\phi(g),\tau(g)\bigr) \in\Z\oplus T$. Then any character $\chi: G \to \C^{\times}$ factors through $\chi=(\sigma,\lambda)$ where $\sigma:T\longrightarrow\C^{\times}$ is a character and $\lambda \in \C^{\times}$. 

Let $G=\pi_1(X)$. For each $\sigma:T\longrightarrow\C^{\times}$, define $\Phi_\sigma:\C G\longrightarrow \Lambda=\C[t^{\pm1}]$ by $\Phi_\sigma(g)=\sigma(\tau(g))t^{\phi(g)}.$ We define $\Lambda_{\sigma}$ to be $\Lambda_{\sigma}=\Lambda$, regarded as a right $\mathbb CG$-module by $a\cdot g=a\Phi_{\sigma}(g)$. Note that $\Lambda_{\sigma}$ is also a left $\Lambda$-module through self-multiplication. Let $X_{\phi}$ be the infinite cyclic cover of $X$ and let $C_*^{\sigma}(X_\phi;\mathbb C)=\Lambda_{\sigma}\otimes_{\mathbb CG} C_*(\widetilde X;\mathbb C)$. The $\sigma$-twisted Alexander module $A_{X}^{\sigma}$ is defined to be $H_1(C_*^{\sigma}(X_\phi;\mathbb C))$ and the $\sigma$-twisted Alexander polynomial is the order of $A_X^{\sigma}$. 

For $\lambda\in\C^\times$, define $\C_\lambda$ to be the left $\Lambda$-module $\Lambda/(t-\lambda)=\C$ on which $t$ acts by $t\cdot z=\lambda z.$ For each character $\sigma: T \to \C^{\times}$, define $\C_{\sigma,\lambda}$ to be the left $\C G$-module $\C$ with $g\cdot z=\sigma(\tau(g))\lambda^{\phi(g)}\cdot z.$ Then we have the following lemma.

\begin{lemma}\label{Computing group cohomology using infinite cyclic cover}
Let $X$ be a compact connected orientable three-manifold with torus boundary with $G=\pi_1(X)$, and let $X_\phi$ be the infinite cyclic cover of $X$. Then we have
\begin{equation}\label{group cohomology using infinite cyclic cover}
H^1(G;\C_{\sigma,\lambda})
\cong
H^1\left(\Hom_{\Lambda}(C_*^{\sigma}(X_\phi;\C),\C_\lambda)\right).
\end{equation}
\end{lemma}

\begin{proof}
By definition, $C_{*}^{\sigma}(X_\phi;\C)=\Lambda_{\sigma}\otimes_{\C G} C_{*}(\widetilde{X};\C)$, which is naturally a chain complex of left $\Lambda$-modules through self-multiplication on $\Lambda_{\sigma}$.

For a left $\C G$-module $P$ and a left $\Lambda$-module $M$,
there is a natural isomorphism called Tensor-Hom adjunction given by 
$$
\Hom_\Lambda\left(\Lambda_{\sigma}\otimes_{\C G}P,M\right)
\cong
\Hom_{\C G}(P,M),
$$
where $M$ is regarded as a $\C G$-module via the composition of $\Phi_{\sigma}:\C G\longrightarrow\Lambda$ and the action of $\Lambda$. 

Applying this adjunction degree by degree to $P_*=C_*(\widetilde X;\C),~M=\C_\lambda,$ gives an isomorphism of cochain complexes
$$
\Hom_\Lambda\left(\Lambda_{\sigma}\otimes_{\C G}C_*(\widetilde X;\C),\C_\lambda\right)
\cong
\Hom_{\C G}\left(C_*(\widetilde X;\C),\C_{\sigma,\lambda}\right).
$$
which gives 
$$
\Hom_\Lambda(C_{*}^{\sigma}(X_\phi;\C),\C_\lambda)
\cong \Hom_{\C G}(C_*(\widetilde X;\C),\C_{\sigma,\lambda}).
$$
Then we apply lemma \ref{Computing group cohomology using universal cover} by setting $M=\C_{\sigma,\lambda}$ in \eqref{group cohomology using universal cover} to give 
$$
\begin{aligned}
H^1(G;\C_{\sigma,\lambda})
&\cong H^1\left(\Hom_{\C G}(C_*(\widetilde X;\C),\C_{\sigma,\lambda})\right)\\
&\cong H^1\left(\Hom_\Lambda(C_{*}^{\sigma}(X_\phi;\C),\C_\lambda)\right).
\end{aligned}
$$
This completes the proof. 
\end{proof}

We now prove the following important theorem. 

\begin{theorem}\label{equivalence of H^1 and twisted Alexander}
Let $X$ be a compact connected orientable three-manifold with torus boundary with $G=\pi_1(X)$ and $H_{1}(X) \cong \Z \oplus T$ where $T$ is the finite torsion part. Let $\chi=(\sigma,\lambda)$ be the decomposition of a character $\chi: G \to \C^{\times}$, and let $\Delta_{X}^{\sigma}(t)$ be the associated twisted Alexander polynomial. Suppose $\chi$ is nontrivial,i.e. $(\sigma,\lambda) \neq (1,1)$, then we have
$$
H^1(G;\mathbb C_{\sigma,\lambda})\neq 0\quad\Longleftrightarrow\quad
\Delta_{X}^\sigma(\lambda)=0.
$$
\end{theorem}

\begin{proof}
Set $Q_*=C_*^\sigma(X_\phi;\mathbb C).$ Since $Q_*$ is a chain complex of free $\Lambda$-modules and $\Lambda=\C[t^{\pm 1}]$ is a PID, the universal coefficient theorem gives a natural short exact sequence
$$
0
\longrightarrow
\operatorname{Ext}^1_\Lambda\left(H_0(Q_*),\mathbb C_\lambda\right)
\longrightarrow
H^1\left(\operatorname{Hom}_\Lambda(Q_*,\mathbb C_\lambda)\right)
\longrightarrow
\operatorname{Hom}_\Lambda\left(H_1(Q_*),\mathbb C_\lambda\right)
\longrightarrow 
0.
$$
Using $H_1(Q_*)=A_X^\sigma$ and lemma \ref{Computing group cohomology using infinite cyclic cover}, this becomes
$$
0
\longrightarrow
\operatorname{Ext}^1_\Lambda\left(H_0\bigl(C_*^\sigma(X_\phi;\mathbb C)\bigr),
\mathbb C_\lambda\right)
\longrightarrow 
H^1(G;\mathbb C_{\sigma,_\lambda})
\longrightarrow
\operatorname{Hom}_\Lambda\left(A_X^\sigma,\mathbb C_\lambda\right)
\longrightarrow 
0.
$$
Therefore, whenever $\operatorname{Ext}^1_\Lambda\left(H_0\bigl(C_*^\sigma(X_\phi;\mathbb C)\bigr),
\mathbb C_\lambda\right)=0,$ there is a natural isomorphism
$$
H^1(G;\mathbb C_{\sigma,_\lambda})
\cong
\operatorname{Hom}_\Lambda\left(A_X^\sigma,\mathbb C_\lambda\right).
$$

We now prove that when $(\sigma,\lambda) \neq (1,1)$, the Ext term indeed vanishes. Suppose that $G$ is generated by $x_1,\ldots,x_m.$ The degree-one boundary map is $\partial_1:\Lambda^m\longrightarrow\Lambda,$ where $\partial_1(e_j)=\sigma(\tau(x_j))t^{\phi(x_j)}-1.$ Consequently,
$$
H_0(C_*^\sigma)
\cong
\frac{\Lambda}{\left(\sigma(\tau(g))t^{\phi(g)}-1\mid g\in G\right)}.
$$

Suppose that $\sigma$ is nontrivial. Choose $u\in T$ such that $\sigma(u)\neq1.$
Since the projection of $G$ onto $H_1(X;\Z)\cong\Z\oplus T$ is surjective, there exists $g\in G$ whose homology class is $[g]=(0,u).$ Thus $\phi(g)=0,~\tau(g)=u.$ The corresponding relation in $H_0(C_*^\sigma)$ is $\sigma(\tau(g))t^{\phi(g)}-1=\sigma(u)-1.$
Since $\sigma(u)-1$ is a nonzero complex number, it is a unit in
$\Lambda$. Therefore, $H_0(C_*^\sigma)=0.$ It follows immediately that $\Ext^1_\Lambda\left(H_0(C_*^\sigma),\C_\lambda\right)=0$ for every $\lambda\in\C^\times$.

Suppose now that $\sigma$ is trivial. Then

$$
H_0(C_*^1)\cong\frac{\Lambda}{\left(t^{\phi(g)}-1\mid g\in G\right)}.
$$
Every polynomial $t^{\phi(g)}-1$ is divisible by $t-1$. Since $\phi$ is surjective, there exists $g_0\in G$ such that $\phi(g_0)=1.$ For this element,
$t^{\phi(g_0)}-1=t-1.$ Therefore, $H_0(C_*^1)\cong\Lambda/(t-1).$ To calculate the relevant $\Ext$ term, consider the free resolution of $\Lambda/(t-1)$
$$
0\longrightarrow\Lambda\xrightarrow{\,t-1\,}\Lambda\longrightarrow\Lambda/(t-1)\longrightarrow0.
$$
Since Hom functor is left exact, applying $\Hom_\Lambda(-,\C_\lambda)$ to the short exact sequence above gives
$$
0\longrightarrow\Hom_\Lambda(\Lambda/(t-1),\C_\lambda)\longrightarrow\C_\lambda
\xrightarrow{\,\lambda-1\,}\C_\lambda.
$$
On $\C_\lambda$, multiplication by $t-1$ is just multiplication by
$\lambda-1$, which is invertible when $\lambda \neq 1$. Hence $\Ext^1_\Lambda(\Lambda/(t-1),\C_\lambda)\cong\frac{\C_\lambda}{(\lambda-1)\C_\lambda}=0$ when $\lambda \neq 1$. 

Now by the definition of twisted Alexander polynomial as presented in section \ref{twisted Alexander module and polynomials}, we write
$$
A_X^\sigma
\cong
\bigoplus_{j=1}^s \Lambda/(f_j(t)),
$$
and hence $\Delta_{X}^\sigma(t) \doteq \prod_{j=1}^s f_j(t).$ For each summand, one has $\operatorname{Hom}_\Lambda\left(\Lambda/(f_j(t)),\mathbb C_\lambda\right) \neq 0$ if and only if $f_j(t)$ acts trivially on $\mathbb C_\lambda$. Since
$t$ acts as $\lambda$, this is equivalent to $f_j(\lambda)=0.$ Therefore,
$\operatorname{Hom}_\Lambda\left(A_X^\sigma,\mathbb C_\lambda\right)\neq 0$
if and only if at least one $f_j(\lambda)$ vanishes. Equivalently,
$$
\operatorname{Hom}_\Lambda\left(A_X^\sigma,\mathbb C_\lambda\right)
\neq 0
\quad\Longleftrightarrow\quad\Delta_{X}^\sigma(\lambda)=0.
$$
Since the Ext term vanishes, we therefore obtain
$$
H^1(G;\mathbb C_{\sigma,\lambda})\neq 0\quad\Longleftrightarrow\quad
\Delta_{X}^\sigma(\lambda)=0.
$$
\end{proof}

Collecting all of the previous results, we can prove the following twisted version of Burde-de Rham(cf. \cite{HeusenerPorti2005}, Corollary 5.1). 

\begin{theorem}\label{twisted Burde de Rham theorem}
Let $X$ be a compact connected three-manifold with torus boundary
and fix a projection of $G=\pi_1(X)$ onto $H_1(X;\Z)\cong T\oplus \Z $ given by $(\tau,\phi)$. For any character $\alpha:G \longrightarrow \C^{\times}$, we write $\alpha=(\sigma,\lambda)$ so that $\alpha(g)=\sigma(\tau(g))\lambda^{\phi(g)}$ for any $g \in G$. Assume $\alpha^2\neq1.$ Then the following are equivalent:

(1). There exists a nonabelian reducible representation $\rho:G\longrightarrow \operatorname{SL}(2,\C)$ whose diagonal characters are $\alpha, \alpha^{-1}.$

(2). $\Delta^{\sigma^2}_{X}(\lambda^2)=0.$
\end{theorem}

\begin{proof}
    Combing Theorem \ref{Cohomological criterion} and Theorem \ref{equivalence of H^1 and twisted Alexander} yields the desired result. 
\end{proof}

We now consider the analogous statement for 2-bridge link complements in $\mathbb{S}^3$. Let $L$ be a 2-bridge link in $\mathbb{S}^3$ and let $M=\mathbb{S}^3-N(L)$ be the complement of some neighborhood of $L$ in $\mathbb{S}^3$. Then $H_1(M;\Z)\cong \Z\oplus\Z $ has no torsion, and $\Z\oplus\Z$ is generated by a pair of preferred meridians on each of the cusp. We fix a projection of $G=\pi_1(M)$ onto $H_1(M;\Z)\cong \Z\oplus\Z $ given by $(\phi_1,\phi_2)$, then any character $\alpha:\pi_1(M) \longrightarrow \C^{\times}$ factors through $(\lambda_1,\lambda_2)$ via $\alpha(g)=\lambda_1^{\phi_1(g)}\lambda_2^{\phi_2(g)}$. Let $A_M,~\Delta_M(t_1,t_2)$ be the Alexander module and Alexander polynomial of $M$ respectively, as defined in \ref{def of Alexander for 2-components}. Then the 2-bridge version of the Burde-de Rham criterion is stated as follows. 

\begin{theorem}\label{Burde de Rham theorem for 2-bridge link}
Let $L$ be a 2-bridge link in $\mathbb{S}^3$ and let $M=\mathbb{S}^3-N(L)$ be the complement of some neighborhood of $L$ in $\mathbb{S}^3$. Fix a projection of $G=\pi_1(M)$ onto $H_1(M;\Z)\cong \Z\oplus\Z $ given by $(\phi_1,\phi_2)$. For any character $\alpha:G \longrightarrow \C^{\times}$, we write $\alpha=(\lambda_1,\lambda_2)$ so that $\alpha(g)=\lambda_1^{\phi_1(g)}\lambda_2^{\phi_2(g)}$ for any $g \in G$. Assume $\alpha^2\neq1.$ Then the following are equivalent:

(1). There exists a nonabelian reducible representation $\rho:G\longrightarrow \operatorname{SL}(2,\C)$ whose diagonal characters are $\alpha, \alpha^{-1}.$

(2). $\Delta_{M}(\lambda_1^2,\lambda_2^2)=0.$
\end{theorem}

\begin{proof}
Theorem \ref{Cohomological criterion} and lemma \ref{Computing group cohomology using universal cover} can be applied directly to the situation where $M$ is a 2-bridge link complement. For lemma \ref{Computing group cohomology using infinite cyclic cover}, set $\Lambda=\C[t_1^{\pm 1},t_2^{\pm 1}]$ and $C_*(M_\phi;\mathbb C)=\Lambda\otimes_{\mathbb CG} C_*(\widetilde M;\mathbb C)$. Then applying the same arguments in lemma \ref{Computing group cohomology using infinite cyclic cover}, we obtain the similar result for $M$: $H^1(G;\C_{\lambda_1,\lambda_2})
\cong H^1\left(\Hom_{\Lambda}(C_*(M_\phi;\C),\C_{\lambda_1,\lambda_2})\right)$.   For Theorem \ref{equivalence of H^1 and twisted Alexander}, we can apply similar arguments with the following modifications. When $M$ is a 2-bridge link complement, $\Lambda=\C[t_1^{\pm 1},t_2^{\pm 1}]$ is no longer a PID and we can not directly apply the universal coefficient theorem. Instead, we apply the universal coefficient spectral sequence to obtain the exact sequence 

$$
\begin{aligned}
0\longrightarrow&
\operatorname{Ext}_\Lambda^1(H_0(Q_*),\C_{\lambda_1,\lambda_2})
\longrightarrow H^1(M;\C_{\lambda_1,\lambda_2})\\
\longrightarrow&
\operatorname{Hom}_\Lambda(A_M,\C_{\lambda_1,\lambda_2})
\longrightarrow
\operatorname{Ext}_\Lambda^2(H_0(Q_*),\C_{\lambda_1,\lambda_2})
\longrightarrow H^2(M;\C_{\lambda_1,\lambda_2}).
\end{aligned}
$$
where $Q_{*}=C_*(M_\phi;\mathbb C)$. By similar computations in the proof of Theorem \ref{equivalence of H^1 and twisted Alexander}, we can show that the two Ext terms in the sequence vanish when $(\lambda_1,\lambda_2)\neq (1,1)$ and hence  $H^1(M;\C_{\lambda_1,\lambda_2}) \cong \operatorname{Hom}_\Lambda(A_M,\C_{\lambda_1,\lambda_2})$. By remark \ref{Alexander poly for 2-bridge link}, $A_M \cong \Lambda/(\Delta_M)$ when $M$ is a 2-bridge link in $\mathbb{S}^3$. Therefore, $\operatorname{Hom}_\Lambda(A_M,\C_{\lambda_1,\lambda_2}) \neq 0$ if and only if $\Delta_{M}(\lambda_1,\lambda_2)=0$. Combing this result with Theorem \ref{Cohomological criterion} yields the desired result.  
\end{proof}

We now use these results to study the reducible intersections using twisted Alexander polynomials. Let $M$ be a two bridge link complement in $\mathbb{S}^3$ with torus boundaries $T_1,T_2$ and let $X_0(M)$ be a canonical surface in the character variety of $M$. Let $(\mu_1,\lambda_1),(\mu_2,\lambda_2)$ be preferred pairs of meridians and longitudes for $T_1,T_2$ respectively. Then $H_1(M;\Z) \cong \Z \oplus \Z$, with free generators given by $[\mu_1], [\mu_2]$. We perform Dehn-fillings on $T_1$ with slopes $\gamma=\frac{p}{q},\delta=\frac{r}{s}$ to obtain $M_\gamma,M_{\delta}$. Assume that the slopes are nonzero, i.e. $p\neq 0, r\neq 0$, and the linking number of the two boundary components is 0. Then $H_1(M_\gamma;\Z) \cong H_1(M;\Z)/\langle p[\mu_1]\rangle \cong \Z/|p|\Z \oplus \Z$ where the generators are given by $[\overline{\mu_1}], [\overline{\mu_2}]$ respectively and $\overline{\mu_i}$ is the image of $\mu_i$ under the surjective quotient map $\iota_{\gamma}: \pi_1(M) \longrightarrow \pi_1(M_\gamma)$. Similar results also hold for $H_1(M_\delta;\Z) \cong \Z/|r|\Z \oplus \Z$.

\begin{remark}\label{linking number 0 guarantees free generator}
If the linking number of the two boundary components is $k\neq 0$, then $H_1(M_\gamma;\Z) \cong H_1(M;\Z)/\langle p[\mu_1]+q[\lambda_1]\rangle \cong \langle p[\mu_1]+qk[\mu_2]\rangle$. Then the order of the torsion part of $H_1(M_\gamma;\Z)$ may not be $|p|$ and $[\overline{\mu_2}]$ may not be the generator for the free part. 
\end{remark}

Now define $C_\gamma, C_\delta$ as in section 2.3. Let $\chi \subset C_\gamma \cap C_\delta \subset X_0(M)$ be a reducible character with representative $\rho: \pi_1(M) \longrightarrow \operatorname{SL}(2,\C)$. Then $\rho=\iota_{\gamma,*}(\rho_{\gamma})=\rho_{\gamma}\circ \iota_{\gamma}$ for some $\rho_{\gamma}:\pi_1(M_{\gamma})\longrightarrow \operatorname{SL}(2,\C)$. Similarly, $\rho=\iota_{\delta,*}(\rho_{\delta})=\rho_{\delta}\circ \iota_{\delta}$ for some $\rho_{\delta}:\pi_1(M_{\delta})\longrightarrow \operatorname{SL}(2,\C)$. Since $\rho$ is reducible, both $\rho_\gamma,\rho_\delta$ are reducible. We choose $\rho$(resp. $\rho_\gamma,\rho_\delta$) to be upper-triangular with diagonal entries $(\alpha,\alpha^{-1})$(resp. $(\alpha_{\gamma},\alpha_{\gamma}^{-1}),(\alpha_{\delta},\alpha^{-1}_{\delta})$), where $\alpha:\pi_1(M) \longrightarrow \C^{\times}$ is a group homomorphism and $\alpha=\iota_{\gamma,*}(\alpha_{\gamma})=\alpha_{\gamma}\circ \iota_{\gamma}=\iota_{\delta,*}(\alpha_{\delta})=\alpha_{\delta}\circ \iota_{\delta}$. We further assume that $\chi \subset C_1\cap C_2$ where $C_1\subset C_{\gamma}, C_2 \subset C_{\delta}$ are two irreducible curves, each containing a character of an irreducible representation. We also assume that $\chi \notin \phi^{-1}(\Sigma)$ where $\phi$ is the boundary restriction map onto $X(T_1)$ and $\Sigma$ consists of the four singular points in $X(T_1)$. By our assumptions on $\chi$ and \cite{Chin_Reid_Stover} Proposition 2.11, $\rho_\gamma,\rho_\delta$(and hence $\rho$) can be taken to have nonabelian image and $\chi \notin \phi^{-1}(\Sigma)$ guarantees that $\alpha_\gamma,\alpha_\delta,\alpha \neq \pm 1$. Then we have the following.

\begin{proposition}\label{reducible intersection and common root}
Let $M$ be a two bridge link complement in $\mathbb{S}^3$ such that the linking number of the two boundary components is 0. Let $\gamma=\frac{p}{q}\neq 0,\delta=\frac{r}{s}\neq0$ and let $C_1\subset C_{\gamma}, C_2 \subset C_{\delta}$ be two irreducible curves, each containing a character of an irreducible representation. If $\chi$ is a reducible character and $\chi \in C_1\cap C_2\setminus \phi^{-1}(\Sigma)$, then $\chi$ detects a common root $\lambda^2$ for $\Delta_{M_\gamma}^{\sigma^2_{\gamma}}$ and $\Delta_{M_\delta}^{\sigma^2_{\delta}}$, where $\sigma_{\gamma}(\overline{1})=\sigma_{\delta}(\overline{1})=\xi$ and $\xi$ is some $d$-th root of unity for $d=\operatorname{gcd}(|p|,|r|)$. 
\end{proposition}

\begin{proof}
Since the linking number is 0, from the discussion above, we can fix the isomorphism $H_1(M_\gamma,\Z) \cong \Z/|p|\Z \oplus \Z$ with generators $[\overline{\mu_1}], [\overline{\mu_2}]$. Then $\alpha_{\gamma}=(\sigma_\gamma,\lambda_{\gamma})$ where $\lambda_\gamma=\alpha_{\gamma}(\overline{\mu_2})$ and $\sigma_\gamma(\overline{1})=\alpha_{\gamma}(\overline{\mu_1})$. Similar results hold for $\alpha_{\delta}=(\sigma_\delta,\lambda_{\delta})$. Since $\alpha=\alpha_{\gamma}\circ \iota_{\gamma}=\alpha_{\delta}\circ \iota_{\delta}$, we have $\alpha(\mu_2)=\alpha_{\gamma}\circ \iota_{\gamma}(\mu_2)=\alpha_{\delta}\circ \iota_{\delta}(\mu_2)$, i.e. $\lambda_\gamma=\lambda_\delta=\lambda$, and $\alpha(\mu_1)=\alpha_{\gamma}\circ \iota_{\gamma}(\mu_1)=\alpha_{\delta}\circ \iota_{\delta}(\mu_1)$, i.e. $\sigma_\gamma(\overline{1})=\sigma_\delta(\overline{1})=\xi$. Then consider $\rho_\gamma,\rho_\delta$. From the discussion above and Theorem \ref{twisted Burde de Rham theorem}, we have $\Delta_{M_\gamma}^{\sigma^2_{\gamma}}(\lambda^2)=\Delta_{M_\delta}^{\sigma^2_{\delta}}(\lambda^2)=0$ and $\sigma_{\gamma}(\overline{1})=\sigma_{\delta}(\overline{1})=\xi$. $\xi$ is some $d$-th root of unity for $d=\operatorname{gcd}(|p|,|r|)$ because the domain for $\sigma_\gamma, \sigma_\delta$ are $\Z/|p|\Z, \Z/|r|\Z$ respectively. 
\end{proof}

\begin{remark}
If the linking number of the two boundary components is not 0, then as suggested in remark \ref{linking number 0 guarantees free generator}, $[\overline{\mu_2}]$ may not be the generator for the free part of $H_1(M_\gamma,\Z)$(resp. $H_1(M_\delta,\Z)$). Then $H_1(M_\gamma,\Z)$ and $H_1(M_\delta,\Z)$ may no longer have compatible splitting and a common root may not be detected. 
\end{remark}

Finally, we use the Burde-de Rham criterion to study reducible intersection for the Whitehead link. If we perform a Dehn-filling on $T_1$ of the Whitehead link complement with slope $\gamma=\frac{1}{n}$, then we will obtain a twist knot complement $K_n$ in $\mathbb{S}^3$, where the 1st homology group has no torsion. We assume $K_n$ is hyperbolic. Then From \cite{MPV2011}, the character variety $X(K_n)$ contains exactly two irreducible components and one of them is the canonical curve. We prove the following.    

\begin{corollary}\label{never intersect at reducible for whitehead}
Let $M$ be the Whitehead link complement in $\mathbb{S}^3$ and let $\gamma=\frac{1}{n}, \delta=\frac{1}{m},~n\neq m$ such that $M_\gamma=K_n, M_{\delta}=K_m$ are both hyperbolic. Let $C_1\subset C_{\gamma}, C_2 \subset C_{\delta}$ be the canonical curves of $X(K_n)$ and $X(K_m)$ respectively, when embedded in $X_0(M)$. Then $C_1$ and $C_2$ can never intersect at a reducible character $\chi \notin \phi^{-1}(\Sigma)$.  
\end{corollary}

\begin{proof}
The boundary components of the Whitehead link have linking number 0. Suppose $C_1,C_2$ intersect at a reducible character $\chi \notin \phi^{-1}(\Sigma)$, then from Proposition \ref{reducible intersection and common root}, $\chi$ will detect common roots of the ordinary Alexander polynomials for $K_n,K_m$. But the ordinary Alexander polynomial for $K_n$ is given by $\Delta_{K_n}(t)\doteq -nt^2+(2n+1)t-n$. It is easy to check that $\Delta_{K_n}(t), \Delta_{K_m}(t)$ have no common root, which leads to the contradiction.
\end{proof}

\bibliographystyle{spmpsci}
\bibliography{Ref.bib}

\end{document}